\documentclass[12pt,reqno]{amsart}
\usepackage{amsmath,amssymb,enumerate}
\usepackage[all]{xy}
\usepackage{bm}

\newtheorem{tm}{Theorem}[section]
\newtheorem{lm}[tm]{Lemma}
\newtheorem{co}[tm]{Corollary}

\newtheorem{df}[tm]{Definition}
\newtheorem{pr}[tm]{Proposition}

\newtheorem{ex}[tm]{Example}

\makeatletter
\newcommand{\subscripts}[3]{%
  \@mathmeasure\z@\displaystyle{#2}%
  \global\setbox\@ne\vbox to\ht\z@{}\dp\@ne\dp\z@
  \setbox\tw@\box\@ne
  \@mathmeasure4\displaystyle{\copy\tw@_{#1}}%
  \@mathmeasure6\displaystyle{{#2}_{#3}}%
  \dimen@-\wd6 \advance\dimen@\wd4 \advance\dimen@\wd\z@
  \hbox to\dimen@{}\mathop{\kern-\dimen@\box4\box6}%
}
\makeatother

\newcommand{\A}{\mathcal{A}}

\newcommand{\g}{\mathfrak{g}}
\newcommand{\nn}{\nonumber}

\newcommand{\LL}{\mathcal{L}}

\newcommand{\h}{\mathrm{H}}
\newcommand{\e}{\mathrm{e}}

\newcommand{\ve}{\varepsilon}
\newcommand{\dis}{\displaystyle}
\newcommand{\Hom}{\mathrm{Hom}}

\newcommand{\del}{\partial}
\newcommand{\ol}{\overline}

\newcommand{\va}{\mathfrak{a}}
\newcommand{\Span}{\mathrm{Span}_{\mathbb{R}}}

\newcommand{\la}{\langle}
\newcommand{\asy}{\rho_{\mathbb{R}}(\gamma_p)}
\newcommand{\La}{\langle\!\langle}
\newcommand{\ra}{\rangle}

\newcommand{\Ra}{\rangle\!\rangle}

\allowdisplaybreaks[4]

\makeatletter
 
 \@addtoreset{equation}{section}
\makeatother

\begin{document}
\title[Edgeworth expansions on covering graphs]{Edgeworth expansions 
for centered random walks on covering graphs  of
polynomial volume growth}
\author[Ryuya Namba]{Ryuya Namba}
\date{\today}
\address{Department of Mathematical Sciences,
College of Science and Engineering,
Ritsumeikan University, 1-1-1, Noji-higashi, Kusatsu, 525-8577, Japan}
\email{{\tt{rnamba@fc.ritsumei.ac.jp}}}

\begin{abstract}
Edgeworth expansions for  random walks on covering graphs 
with groups of polynomial volume growths are obtained
under a few natural assumptions. The coefficients appearing 
in this expansion depends on not only geometric features of the underlying graphs but also
the modified harmonic embedding of the graph into a certain nilpotent Lie group. 
Moreover, we apply the rate of convergence in
Trotter's approximation theorem  to establish the
Berry--Esseen type bound for the random walks. 
\end{abstract}

\maketitle 
%
%


\section{{\bf Introduction and main results}}
It is well-known that {\it central limit theorems} (CLTs, in short) are 
fundamental topics among a lot of branches of mathematics including e.g.,
probability theory, graph theory, geometric analysis and harmonic analysis. 
They have been studied under various settings intensively and extensively. 
In particular,  long time asymptotics of random walks on infinite graphs 
are affected by not only properties of random walks themselves but also
geometric features of underlying spaces such as periodicities and volume growths. 

We mainly deal with {\it covering graphs} as our discrete models of interest in the 
present paper. 
Given a finitely generated group $\Gamma$, 
an infinite connected graph $X=(V, E)$ is called a {\it $\Gamma$-covering graph}
if $X$ admits free $\Gamma$-actions on itself and the quotient base graph 
$X_0:=\Gamma \backslash X=(V_0, E_0)$ is a finite graph. 
In particular, $X$ is called a $\Gamma$-{\it crystal lattice} if 
$\Gamma$ is abelian, which includes classical periodic lattices such as 
square lattices, triangular lattices and hexagonal lattices. 
We here emphasize that the study of covering graphs is closely related to 
e.g., crystallography, tilling theory and material sciences. See \cite{S}
for more details with extensive references therein.

There exist several papers in which some long time asymptotics 
of random walks on crystal lattices are discussed. 
In general, it is not possible to apply the usual techniques to 
obtain such asymptotics directly to the case where the underlying space is a graph, 
since the notion of  scale changes on graphs is not defined properly. 
Kotani and Sunada studied long time
asymptotics of random walks on a crystal lattice $X$ by placing a special emphasis on the geometric feature of $X$ in e.g. \cite{KS00-CMP, KS00-TAMS}. 
Especially, they introduced the notion of harmonic realization of $X$, 
which is a discrete harmonic map  $\Phi_0 : X \to \Gamma \otimes \mathbb{R}=\mathbb{R}^d$  
to characterize an equilibrium configuration of $X$.
Among their studies, they developed a new mathematical branch
called the {\it discrete geometric analysis} and it has been effectively applied 
to capture some asymptotics of random walks on $X$. 
See e.g., \cite{KS06, IKK}. 

On the other hand, there exist several studies on long time asymptotics of 
symmetric or centered random walks on  finitely generated
groups. We know  that 
the notion of volume growth of groups
plays a crucial role
in obtaining such results.
On the other hand, it is not possible to characterize
a finitely generated group itself in terms of its volume growth. 
However, there is a celebrated theorem of Gromov (cf.\,\cite{Gromov})
 on a characterization of groups of polynomial volume growth, which asserts that 
such a group is {\it virtually nilpotent}.  
Namely, there exists a nilpotent subgroup of $\Gamma$ with finite index.
Hence, a number of results on long time asymptotics of 
random walks on nilpotent groups have been known. 
We refer to 
\cite{Raugi, Pap1, A1} for CLTs on nilpotent Lie groups and 
\cite{Bre, Hough} for 
local CLTs on 
nilpotent Lie groups. 

Afterwards, a hybrid model between crystal lattices and groups of polynomial volume growths is 
introduced. Namely, we consider a $\Gamma$-covering graph $X=(V, E)$ of a finite graph $X_0=(V_0, E_0)$
whose covering transformation group $\Gamma$ is finitely generated
and of polynomial volume growth. 
By virtue of 
the Gromov theorem, we may assume that $\Gamma$ is nilpotent 
without loss of generality. So we simply call $X$ a {\it $\Gamma$-nilpotent covering graph}
in what follows.  
It is known that there exists a connected and simply connected 
nilpotent Lie group $(G, \bullet)$ in which $\Gamma$ is a cocompact lattice
(see \cite{Malcev}). 
Therefore, we come to define the notion of harmonic realizations $\Phi_0 : X \to G$
in the same way as that of crystal lattices $\Phi_0 : X \to \mathbb{R}^d$.  
We note that the harmonicity of $\Phi_0 : X \to G$ determines the most
natural configurations of $X$ corresponding to components 
of $G/[G, G]$, the abelianization of $G$. 
See \eqref{modified harmonic} for its precise definition. 
Ishiwata studied 
long time asymptotics of symmetric random walks on $X$ and 
obtained a CLT for them. Moreover, a rate of convergence of 
the $n$-step transition probability $p(n, x, y)$ to the heat kernel on $G$ 
was obtained in a geometric point of view. 
See \cite{Ishiwata, Ishiwata2}. 
Ishiwata, Kawabi and Namba obtained in their successive papers \cite{IKN, IKN2}
that functional CLTs hold for even non-symmetric random walks on $X$
under some natural conditions, which give the process-level convergences 
of such random walks to certain diffusion processes on $G$. 

For an edge $e \in E$, 
we denote by $o(e)$, $t(e)$ and $\ol{e}$ 
the origin, the terminus and the inverse edge of $e$, respectively. 
We put $E_x=\{e \in E \, | \, o(e)=x\}$ for $x \in V.$
Consider a transition probability $p : E \to [0, 1]$ satisfying
$\sum_{e \in E_x}p(e)=1$ for $x \in V$ and $p(e)+p(\ol{e})>0$ for $e \in E$.
The transition probability $p$ is supposed to be invariant under $\Gamma$-actions, that is, 
$p(\gamma e)=p(e)$ for $\gamma \in \Gamma$ and $e \in E$. 
Then,  $p$ yields 
an $X$-valued random walk
$(\Omega_x(X), \mathbb{P}_x, \{w_n\}_{n=0}^\infty)$ starting at $x \in V$, 
where $\mathbb{P}_x$ is 
the probability measure on $\Omega_x(X)$ induced from  $p$
and $w_n(c):=o( e_{n+1})$ for $n \in \mathbb{N} \cup\{0\}$
and $c=(e_1, e_2, \dots, e_n, \dots) \in \Omega_x(X)$. 
Let $\pi : X \to X_0$ be the covering map. 
Then we can consider the transition probability $p$ on $X_0$ defined by 
 $p\big(\pi(e)\big)=p(e)$ for $e \in E$. Namely, $p$ on $X$ is the lift of $p$ on $X_0$.  
Then $p$ on $X_0$ also induces 
 an $X_0$-valued random walk
$(\Omega_{\pi(x)}(X_0), \mathbb{P}_{\pi(x)}, \{w_n=\pi(w_n)\}_{n=0}^\infty)$ starting at $\pi(x) \in V_0$.
In what follows, we assume that the random walk $\{w_n\}_{n=0}^\infty$ on $X_0$
associated with $p$ is {\it irreducible}. Then,  
there exists the unique {\it normalized invariant measure} $m : V_0 \to (0, 1]$ which satisfies
$$
\sum_{x \in V_0}m(x)=1, \qquad \sum_{e \in (E_0)_x}p(\ol{e})m\big(t(e)\big), \qquad x \in V_0,
$$
by applying the Perron--Frobenius theorem. 

Suppose that the nilpotent Lie algebra $\g$ of $G$ 
is decomposed as $\g=\g^{(1)} \oplus \g^{(2)} \oplus \cdots \oplus \g^{(r)}$
for some $r \in \mathbb{N}$ and satisfy some nice conditions (see {\bf (A1)}). 
Denote by $C_\infty(G)$ (resp. $C_\infty(X)$) 
the set of all continuous functions on $G$ (resp. on $X$) vanishing at infinity. 
We also denote by $C_c^\infty(G)$ the set of 
 all compactly supported smooth functions on $G$. 
We define the {\it transition operator} $\LL : C_\infty(X) \to C_\infty(X)$ 
and the {\it approximation operator} $\mathcal{P}^H_\ve : C_\infty(G) \to C_\infty(X), \,\, \ve>0$, by
 $$
 \LL f(x):=\sum_{e \in E_{x}}p(e) f\big(t(e)\big), \qquad 
 \mathcal{P}^H_\ve f(x):=f\Big(\tau_\ve \big(\Phi_0(x)\big)\Big), \qquad 
 x \in X, 
 $$
respectively, where $p$ is a $\Gamma$-invariant non-symmetric transition
probability on $X$, the map $\Phi_0 : X \to G$ 
is the modified harmonic realization of $X$ defined by \eqref{modified harmonic}, 
and $\tau_{n^{-1/2}} : G \to G$ is the dilation of order $n^{-1/2}$. 
See Section 2 for its definition. 

Let $\{\mathfrak{a}_1, \mathfrak{a}_2, \dots, \mathfrak{a}_{d_1}\}$
and $\{\mathfrak{a}_{d_1+1},  \dots, \mathfrak{a}_{d_1+d_2}\}$ 
be fixed bases in $\g^{(1)}$ and $\g^{(2)}$, respectively, 
where $d_k, \, k=1, 2$, is the dimension of $\g^{(k)}$. 
We put 
$$
\sigma_i(\Phi_0):=\sum_{e \in E_0}p(e)m\big(o(e)\big)
\log\Big(\Phi_0(o(\widetilde{e}))^{-1} 
\bullet \Phi_0(t(\widetilde{e}))\Big)\Big|_{\mathfrak{a}_i}, 
\qquad i=1, 2, \dots, d_1,
$$
and 
$$
\beta_i(\Phi_0):=\sum_{e \in E_0}p(e)m\big(o(e)\big)
\log\Big(\Phi_0(o(\widetilde{e}))^{-1} 
\bullet \Phi_0(t(\widetilde{e}))\Big)\Big|_{\mathfrak{a}_i},
\qquad i=d_1+1, \dots, d_1+d_2,
$$
where $\widetilde{e} \in E$ is a lift of $e \in E_0$ to $X$
and $\log : G \to \g$ is the inverse map of $\exp$.

The following is a simple version of the CLT for a non-symmetric random walk
$\{w_n\}_{n=0}^\infty$ 
on the nilpotent covering graph $X$ under the so-called {\it centered condition}, 
whose precise meaning will be given later soon. 

\begin{pr}[cf. {\cite[Theorem 2.1]{IKN}}]
\label{CLT}
Assume the centered condition. Then, for any function
$f \in C_\infty(G)$  and any $t>0$, we have 
\begin{equation}\label{IKN-CLT}
\|\mathcal{L}^{[nt]} \mathcal{P}^H_{n^{-1/2}}f
- \mathcal{P}^H_{n^{-1/2}}\e^{t\mathcal{A}(\Phi_0)}f\|_\infty \to 0 \quad (n \to \infty),  
\end{equation}
where $(\e^{t\mathcal{A}(\Phi_0)})_{t \ge 0}$ is a $C_0$-semigroup whose infinitesimal generator 
$\A(\Phi_0)$ is a second order sub-elliptic operator on $C_c^\infty(G)$ given by
\begin{equation}\label{infinitesimal generator}
\A(\Phi_0)=\frac{1}{2}\sum_{i, j=1}^{d_1}\sigma_i(\Phi_0) \sigma_j(\Phi_0) 
\mathfrak{a}_i \mathfrak{a}_j + 
\sum_{i=d_1+1}^{d_1+d_2}\beta_i(\Phi_0)\mathfrak{a}_i.
\end{equation}
\end{pr}
This proposition means that the sequence of 
discrete semigroups 
generated by the $G$-valued random walk $\{\Phi_0(w_n)\}_{n=0}^\infty$  
 converges 
to a  heat semigroup $(\e^{t\A(\Phi_0)})_{t \ge 0}$ on $G$ under the usual CLT-scaling. 
The point of interest is that the symmetry of the random walk $\{w_n\}_{n=0}^\infty$ implies 
$\beta_i(\Phi_0)=0$ for $i=d_1+1, \dots, d_1+d_2$, whereas the converse does not hold in general. 
Therefore, we observe that each drift coefficient $\beta_i(\Phi_0)$ is 
affected by the non-symmetry of the given random walk.

As a further interesting problem, it is natural to ask whether the precise rate of convergence 
of \eqref{IKN-CLT} is obtained or not. Namely, we would like to find out a function 
$\Psi=\Psi(n)$ satisfying 
\begin{equation}\label{BE-intro}
\|\mathcal{L}^{[nt]} \mathcal{P}^H_{n^{-1/2}}f 
- \mathcal{P}^H_{n^{-1/2}}\e^{t\mathcal{A}(\Phi_0)}f\|_\infty 
\le C \Psi(n), \qquad  n \in \mathbb{N},
\end{equation}
for some positive constant $C>0$ independent of $n$ and any function $f : G \to \mathbb{R}$ 
with some  regularity.  
Such a bound is usually called the {\it Berry--Esseen type bound}
for the discrete semigroup $\mathcal{L}^{[nt]}\mathcal{P}^H_{n^{-1/2}}f$ 
associated with the transition probability $p$. 
A more difficult problem than \eqref{BE-intro} is known as the 
{\it Edgeworth expansions}, which roughly means the asymptotic expansions
on the left-hand side of \eqref{BE-intro}. 
We refer to \cite{GH} for Edgeworth expansions in the CLT for random walks on 
$\mathbb{R}^d$ under some moment conditions, and
\cite{Gotze} for an extension of the Edgeworth expansions to the infinite-dimensional cases. 
Bentkus and Pap \cite{BP} extended 
Raugi's early result of CLTs on nilpotent Lie groups (cf.\,\cite[Theorem 4.5]{Raugi}). 
Namely, they obtained not only 
the Berry--Esseen type bounds but also 
the Edgeworth expansions on nilpotent Lie groups. 
However, their expansions 
are valid only up to the second order. Afterwards, Pap showed that the Edgeworth 
expansions above are also valid up to arbitrary order in \cite{Pap2}, which basically motivates 
our study.

The purpose of the present paper is to obtain the Edgeworth expansions
of a certain class of {\it non-symmetric} random walks $\{w_n\}_{n=0}^\infty$ on $X$
up to arbitrary order
as a refinement of  \eqref{IKN-CLT}, by noting geometric features of $X$ 
including modified harmonicity of realization maps. 
We note that the underlying random walks are supposed to be 
generated by independently and identically distributed (i.i.d., in short) random variables 
in both papers \cite{BP} and \cite{Pap2}. 
On the other hand, the random walks on a nilpotent covering graph $X$ are not always 
generated by i.i.d.\,random variables since $X$ has inhomogeneous local structures.
Therefore, we need careful examinations of techniques used in the present paper. 

Let $\Phi_0 : X \to G$ denote the modified harmonic realization of $X$. 
In the following, we impose two 
technical but quite natural assumptions.

\vspace{2mm}
\noindent
{\bf (A1)}: The nilpotent Lie group $G$ is {\it stratified of step $r \in \mathbb{N}$.} Namely, 
the corresponding Lie algebra 
$(\g, [\cdot, \cdot])$ admits a direct sum decomposition 
$\mathfrak{g}=\mathfrak{g}^{(1)} \oplus \g^{(2)} 
\oplus \dots \oplus \mathfrak{g}^{(r)}$ satisfying that
$
[\mathfrak{g}^{(k)}, \mathfrak{g}^{(\ell)}] \subset \g^{(k+\ell)}$ if $k+\ell \leq r$
and the subspace $\mathfrak{g}^{(1)}$ generates the whole $\mathfrak{g}$ 
in the sense that $\g^{(k)}=[\g^{(1)}, \g^{(k-1)}]$ for $k=2, 3, \dots, r$.

\vspace{2mm}
\noindent
{\bf (A2)}: Random walks on $X$ is {\it centered}. In other words, it holds that the law of large numbers
$$
\lim_{n \to \infty}\frac{1}{n}\log(\Phi_0(w_n))|_{\g^{(1)}}=0, \quad \text{a.s.},
$$
where $\log(\Phi_0(w_n))|_{\g^{(1)}}$ is the projection of $\log(\Phi_0(w_n))$ onto $\g^{(1)}$-component. 

\vspace{2mm}
We  consider the class of 
stratified nilpotent Lie groups, which is more treatable than that of general ones.   
Moreover, this class should be sufficiently rich to reflect the 
geometric structures of general nilpotent Lie groups. 
This is why we may impose {\bf (A1)} in the present paper. 
On the other hand, Proposition \ref{IKN-CLT} also holds without imposing {\bf (A2)}
 (cf.\,\cite[Theorem 2.1]{IKN}). However, we need to consider
a slight extension of the function space to handle 
a certain average of the random walks when {\bf (A2)} is not assumed. 
Therefore, {\bf (A2)} is just imposed to avoid the cumbersomeness of arguments. 
We emphasize that the non-symmetry of the random walks readily implies {\bf (A2)}, 
whereas there exists a non-symmetric random walk on $X$ 
which satisfy {\bf (A2)}. See Section 3 for more details. 
We also emphasize that there is a case where $\beta_i(\Phi_0) \neq 0$ even under {\bf (A2)}.

Then the main result in the present paper is stated as follows:
\begin{tm}[Edgeworth expansion]
\label{main}
Assume {\bf (A1)} and {\bf (A2)}. 
Let $N \ge 2$ and $f \in C_b^{3(N+1)r}(G)$, which is defined by \eqref{C_b}. 
Then we have
\begin{equation}\label{Edgeworth}
\mathcal{L}^{[nt]} \mathcal{P}^H_{n^{-1/2}}f(x)
=\mathcal{P}^H_{n^{-1/2}}\e^{t\mathcal{A}(\Phi_0)}f(x)+\sum_{j=1}^{N-1}\frac{\xi_j}{n^{j/2}}
+O(n^{-N/2}), \qquad x \in V,
\end{equation}
as $n \to \infty$, where the coefficient $\xi_j$, $j=1, 2, \dots, N-1$,
depends on $x$, $G$, $f$, $p$, $\Phi_0$ and $t$. 
\end{tm}

We also obtain the precise representations of 
the coefficients $\{\xi_j\}_{j=1}^{N-1}$ in \eqref{Edgeworth}
by applying the multidimensional version of the 
so-called Euler--Maclaurin formula (see Proposition \ref{Edgeworth-coef}). 
For example, the first coefficient $\xi_1$ is written as 
$$
\xi_1=\sum_{d(I)=3} m^I_{\Phi_0}\int_{\substack{t+s=1 \\ t, s \ge 0}}\iint_{G^2}
\widehat{S}_{(2)}^{I}f(x \bullet x_1 \bullet \bm{1}_G \bullet x_2) \, \nu_t^{\Phi_0}(dx_1)\nu_s^{\Phi_0}(dx_2)\,dt\,ds,
$$
where $\sum_{d(I)=3}$ means the sum which runs over all multi-index of 
weighted length 3, 
the operator $\widehat{S}_{(2)}^{I}$ is the symmetrization of the 
differential operator of the form $X^I$ with respect to the second variable, and 
$m^I_{\Phi_0}$ is a ``moment'' of the $G$-valued random walk $\{\Phi_0(w_n)\}_{n=1}^\infty$ of order $I$. 
Moreover, $(\nu^{\Phi_0}_t)_{t \ge 0}$ is the centered Gaussian semigroup on $G$ 
associated with the infinitesimal generator $\A(\Phi_0)$ defined by \eqref{infinitesimal generator}, 
that is, $(\nu^{\Phi_0}_t)_{t \ge 0}$ is the unique convolution semigroup on $G$ satisfying
$$
\A(\Phi_0)f(x)=\lim_{t \searrow 0}\frac{1}{t}\int_G 
\big(f(x \bullet y) - f(x)\big) \, \nu_t^{\Phi_0}(dy), \qquad f \in \mathrm{Dom}(\A), \, x \in G. 
$$
See Section 2.2 for  general properties of Gaussian semigroups on $G$.
 
We emphasize that the error term $O(n^{-N/2})$ heavily depends on
the sup-norms of the derivatives of the function $f$ of order up to $3(N+1)r$, as is seen in the proof of Theorem \ref{main}.  
What is remarkable is that the Edgeworth expansion holds even when 
the underlying process is not always symmetric and the underlying space has 
non-commutative and inhomogeneous structures. 
As far as we know, there have not been 
such results. Hence, we believe that this result gives a new insight to 
the study of long time asymptotics of random walks.

The rest of the present paper is organized as follows:
In Section 2, we review basics of stratified nilpotent Lie groups including 
several useful formulas such as the Campbell--Baker--Hausdorff formula and 
the stratified Taylor formula. 
Moments of Gaussian measures on nilpotent Lie groups are 
also discussed. 
The basics of discrete geometric analysis on graphs 
due to Kotani and Sunada are presented in Section 3. 
We also introduce the notion of modified harmonic realizations, which 
play a crucial role in the proof of Theorem \ref{main}. 
In Section 4, an ergodic theorem for iterated transition operators is established. 
Before giving a proof of Theorem \ref{main}, we give a rough observation for the validity of
the Edgeworth expansion 
by using explicit calculations of moments of random walks in Section 5. 
We show Theorem \ref{main} and give the representations of the coefficients 
appearing in the expansions (Proposition \ref{Edgeworth-coef}) in Section 6. 
By a simple application of the Edgeworth expansion, we know that 
the Berry--Esseen type bounds are also obtained. In Section 7, we see that 
the Berry--Esseen type bound can be obtained 
by employing an extension of celebrated Trotter's approximation theorem (Theorem \ref{Trotter-refine}). 
Moreover, we see that the Berry--Esseen type bound 
is also established when the realization map of $X$ is not always supposed to be modified harmonic (Theorem \ref{BE-refine}). 
For this sake,  an estimation of the so-called ``corrector,'' which measures the difference between 
a realization and the modified harmonic one, plays an important role in the proof.

\vspace{2mm}
\noindent
{\bf Notations.} Several notations frequently used in the present paper are given below.
\begin{itemize}
\item
For $x \in \mathbb{R}$, the symbol $[x]$ stands for the 
greatest integer less than or equal to $x$. 
\item
The cardinality of a set $A$ is denoted by $|A|$. 
\item
We denote by $C$ a positive constant that may change from line to line.
\item
One writes $f(n)=O(g(n)) \, (n \to \infty)$ if and only if 
there exists a positive real number $M$ and  $n_0 \in \mathbb{N}$ such that  
$|f(n)| \leq Mg(n)$ for $n \ge n_0$. 
\end{itemize}
\section{{\bf Basics of nilpotent Lie groups}}

\subsection{Nilpotent Lie groups}
We give basic notions on nilpotent Lie groups, our continuous models
of interest, and present fundamental calculus on them. 
In general, nilpotent Lie groups are classified by the structures of 
corresponding nilpotent Lie algebras. 
It is well-known that there is a class of nilpotent Lie algebras which admits
finite direct sum decompositions with some suitable Lie brackets.
See {\bf (A1)} in the previous section. 
The class is 
said to be stratified and contains a lot of important examples of 
nilpotent Lie algebras such as the Heisenberg ones. 

\begin{ex}\normalfont
Let $m \in \mathbb{N}$. 
The {\it Heisenberg group} 
$
\mathbb{H}^{2m+1}(\mathbb{R}):
=(\mathbb{R}^{2m+1}, \bullet)$
defined by
$$
\begin{aligned}
&(x_1, x_2, \dots, x_{2m}, x_{2m+1})\bullet (y_1, y_2, \dots, y_{2m}, y_{2m+1})\\
&=\Big(x_1+y_1, \,x_2+y_2, \,\dots,\, x_{2m}+y_{2m}, \,x_{2m+1}+y_{2m+1}+\sum_{k=1}^{m}x_ky_{k+m}\Big)
\end{aligned}
$$
is the simplest example of stratified nilpotent Lie group of step 2. 
Indeed, the corresponding Lie algebra $\g=(\mathbb{R}^{2m+1}, [\cdot, \cdot])$ spanned by $\{\va_1, \va_2, \dots, \va_{2m}, \va_{2m+1}\}$ with
$$
[\va_i, \va_j]:=\begin{cases}
\va_{2m+1} & \text{if }j=i+m, \, i=1, 2, \dots, m\\
-\va_{2m+1} & \text{if }j=i-m, \, i=m+1, m+2, \dots, 2m\\
0 & \text{otherwise}
\end{cases}.
$$
is decomposed as
$
\g=\g^{(1)} \oplus \g^{(2)}\equiv \Span\{\va_1, \va_2, \dots, \va_{2m}\} \oplus \Span\{\va_{2m+1}\}.
$
\end{ex}

Consider a connected and simple connected 
nilpotent Lie group $(G, \bullet)$ whose Lie algebra $(\g, [\cdot, \cdot])$
is stratified.  
Let $d_k$ be the dimension of $\g^{(k)}$ for $k=1, 2, \dots, r$ and $D:=d_1+d_2+\cdots+d_r$.  
We fix a basis $\{\va_1, \va_2, \dots, \va_D\}$ of $\g$ such that 
$\{\va_{d_1+\cdots+d_{k-1}+1}, \dots, \va_{d_1+\cdots+d_{k-1}+d_k}\}$ is a basis of $\g^{(k)}$ for $k=1, 2, \dots, r$,
where $d_0=0$ in convention. 
For $x \in G$, we introduce the {\it global coordinate of the first kind} given by
$x=(x_1, x_2, \dots, x_D)$ when 
$x=\exp\Big(\sum_{i=1}^D x_i \va_i\Big).$  
It is known that there is a family of diffeomorphisms $(\tau_\ve)_{\ve>0}$ behaving as if it were 
scalar multiplications on $G$. 
Let $x \in G$ be written as 
$x=\exp(\va^{(1)}+\va^{(2)}+\cdots+\va^{(r)}), 
\, \va^{(k)} \in \g^{(k)}, \, k=1, 2, \dots, r$. 
Then the {\it dilation map} $\tau_\ve : G \to G$ is defined by
$$
\tau_\ve(x):=\exp(\ve \va^{(1)}+\ve^2 \va^{(2)}+\cdots+\ve^r \va^{(r)}), \qquad \ve \ge 0. 
$$
We also define a continuous function $|\cdot| : G \to [0, \infty)$  by
$$
|x|:=\sum_{i=1}^D|x_i|^{1/\sigma_i}, \qquad x=(x_1, x_2, \dots, x_D) \in G,
$$
where $\sigma_i=k$ if $i=d_1+\cdots+d_{k-1}+1, 
\dots, d_1+\cdots+d_{k-1}+d_k$ for $k=1, 2, \dots, r$. 
 We easily verify that $|\cdot|$ behaves like a norm on $G$, however, it has the {\it homogeneity} in the sense that
 $|\tau_\ve(x)|=\ve|x|$ for $\ve>0$ and $x \in G$. 
 Therefore, $|\cdot|$ is called a {\it homogeneous norm} on $G$. 
 An element $\va \in \g$ may be extended to left and  right invariant $C^\infty$-vector fields  on $G$, that is,
$$
\begin{aligned}
\va f(x)&=\lim_{t \to 0}\frac{1}{t}\Big(f(x \bullet \exp(t\va)) - f(x)\Big), &\quad 
\widehat{\va} f(x)&=\lim_{t \to 0}\frac{1}{t}\Big(f(\exp(t\va)\bullet x) - f(x)\Big)
\end{aligned}$$
for $f \in C^1(G).$

In order to discuss basic calculus such as 
Taylor's formula on the nilpotent Lie group $G$, 
we now introduce the notion of multi-indices.  
We put $\mathcal{I}:=\{I=(i_1, i_2, \dots, i_D) \in (\mathbb{N} \cup \{0\})^D\}$.  
For $I=(i_1, i_2, \dots, i_D) \in \mathcal{I}$
and $x=(x_1, x_2, \dots, x_{D}) \in G$, 
we put $x^I:=x_{1}^{i_1}x_{2}^{i_2} \cdots x_{D}^{i_D}$. 
Similarly, we put 
$\va^I:=\va_1^{i_1}\va_2^{i_2}\cdots \va_D^{i_D}$
for $\va \in  \g$. We define
$|I|:=i_1+i_2+\cdots+i_D$ and $d(I):=\sigma_{1}i_1+\sigma_2 i_2+\cdots+\sigma_D i_D$ 
for  $I=(i_1, i_2, \dots, i_{D}) \in \mathcal{I}$. 
For $I=(i_1, i_2, \dots, i_D)$ and $J=(j_1, j_2, \dots, j_D)$ in $\mathcal{I}$, 
we put $I+J=(i_1+j_1, i_2+j_2, \dots, i_D+j_D)$. We also denote by $[j]$ 
the multi-index with 1 in the $j$-th component and 0 in the others. 

The {\it Campbell--Baker--Hausdorff formula}, 
which is known to be the most important tool in calculus
on nilpotent Lie groups, is given as  
\begin{equation}\label{CBH}
(x\bullet y)^I=x^I+y^I+\sum_{\substack{d(J)+d(K)=d(I) \\ d(J), d(K) \geq 1}}
C_{JK}x^Jy^K, \qquad x, y \in G.
\end{equation}

For a multi-index $I \in \mathcal{I}$, we define the left and right invariant 
differential operator $S^I$ by
$$
S^I:=\frac{1}{|I|!}\sum_{[j_1]+[j_2]+\cdots+[j_{|I|}]=I}\va_{j_1}\va_{j_2}\cdots \va_{j_{|I|}}, \qquad 
\widehat{S}^I:=\frac{1}{|I|!}\sum_{[j_1]+[j_2]+\cdots+[j_{|I|}]=I}
\widehat{\va}_{j_1}\widehat{\va}_{j_2}\cdots \widehat{\va}_{j_{|I|}},
$$
which are regarded as the symmetrizations of 
$\mathfrak{a}^I$ and $\widehat{\mathfrak{a}}^I$, respectively.

We are now ready for stating left stratified Taylor's formula on $G$,
which plays a key role in proving main results (see \cite[Section 1]{FS} 
or \cite[Section 20]{BLU} for details). We note that right stratified Taylor's formula on $G$ is also
stated in the same way as the left one. 

\begin{lm}\label{Taylor}
Let $N \in \mathbb{N} \cup\{0\}$, $f \in C^{N}(G)$ and $x \in G$.
Then we have
$$
\begin{aligned}
f(x \bullet y)=\sum_{d(I) \le N}S^I f(x) y^I + R^f_{N+1}(x, y), \qquad y \in G,
\end{aligned}
$$
where the remainder term $R^f_{N+1}(x, y)$ satisfies 
\begin{equation}\label{Taylor remainder}
|R^f_{N+1}(x, y)| \le 
C|y|^{N+1}\sup\Big\{ |\mathfrak{a}^I f(x \bullet z)| \, : \, d(I)=N+1, \, |z| \le b^{N+1}|y|\Big\}
\end{equation}
for some positive constants $C, \, b>0$ depending only on $G$. 
\end{lm}

For $N=1, 2, 3, \dots$, we define
\begin{equation}\label{C_b}
C_b^N(G):=\Big\{f \in C_b(G) \, \Big| \, 
\va^I f \in C_b(G), \, I \in \mathcal{I}, \, d(I) \leq N\Big\}. 
\end{equation}
Given $f \in C^N_b(G)$, we put
$$
\|D^{N}f\|_{\infty}:=\|f\|_\infty + \sum_{d(I) \le N}\|\va^I f\|_\infty,
$$
where $\|\cdot\|_\infty$ denotes the usual sup-norm. 
The following is useful in the proof of Theorem \ref{main}. 
See \cite[Lemma 1]{Pap2}. 

\begin{lm}\label{estimate of f}
Let $\ell \in \mathbb{N}$ and $I_1, I_2, \dots, I_\ell \in \mathcal{I}$. 
We put $d:=d(I_1)+d(I_2)+\cdots+d(I_\ell)$ and take 
$f \in C_b^{rd}(G)$. Then we have 
$$
|\widehat{S}_{(2\ell)}^{I_\ell} \cdots \widehat{S}^{I_2}_{(4)} \widehat{S}^{I_1}_{(2)}  
f(x_1 \bullet x_2 \bullet \cdots \bullet x_{2\ell+1})|
\le C(G)\Big(1+|x_1 \bullet x_2 \bullet \cdots \bullet x_{2\ell}|^{(r-1)d}\Big)
\|D^{(rd)}f\|_\infty,
$$
where the function $f(x_1 \bullet x_2 \bullet \cdots \bullet x_{2\ell+1})$ is 
understood as $f : G^{2\ell+1} \to \mathbb{R}$ given by 
$$
f(x_1, x_2, \dots, x_{2\ell+1}):=f(x_1 \bullet x_2 \bullet \cdots \bullet x_{2\ell+1})
$$
and $\widehat{S}^I_{(k)}$, $k=1, 2, \dots, 2\ell+1,$ is regarded as a differential 
operator with respect to the $k$-th variable $x_k$. 
\end{lm}

\subsection{ Gaussian measures on nilpotent Lie groups}

In this section, we discuss several properties of Gaussian measures on 
a nilpotent Lie group $G$. 
At the beginning, we give the definition of Gaussian semigroups and measures on $G$. 
We refer to e.g. \cite{Pap3, BP} for more details. 

A family $(\nu_t)_{t \ge 0}$ of probability measures on $G$ is called a {\it convolution semigroup}
if $\nu_s * \nu_t=\nu_{s+t}$ for $s, t \ge 0$ and 
$\lim_{t \searrow 0} \nu_t=\nu_0=\delta_{\bm{1}_G}$ hold, where 
$\delta_{\bm{1}_G}$ denotes the delta measure at the unit $\bm{1}_G$. 
The infinitesimal generator $(\A, \mathrm{Dom}\,(\A))$ 
of the convolution semigroup $(\nu_t)_{t \ge 0}$ is defined by 
$$
\begin{aligned}
\mathrm{Dom}\,(\A)&:=\Big\{f \in \big(C_b(G), \|\cdot\|_\infty\big) \, \Big| \, 
\lim_{t \searrow 0} \frac{1}{t}\int_G \big(f(x \bullet y) - f(x) \big) \, \nu_t(dy) \text{ exists}\Big\}, \\
\A f &:= \lim_{t \searrow 0} \frac{1}{t}\int_G \big(f(x \bullet y) - f(x) \big) \, \nu_t(dy), 
\qquad f \in \mathrm{Dom}\,(\A). 
\end{aligned}
$$

\begin{df}
A convolution semigroup $(\nu_t)_{t \geq 0}$ on $G$ is called a {\it Gaussian semigroup}
if  $\nu_t, \, t > 0,$ is non-degenerate and $t^{-1}\nu_t(G \setminus U) \to 0$ 
as $t \searrow 0$ holds for
any neighborhood $U$ of the unit $\bm{1}_G$. 
A non-degenerate probability measure $\nu$ on $G$ is said to be a {\it Gaussian measure}
if there is a Gaussian semigroup $(\nu_t)_{t \geq 0}$ such that $\nu_1=\nu$. 
\end{df}
\noindent
We emphasize that, if $(\mu_t)_{t \geq 0}$ and $(\nu_t)_{t \geq 0}$ are two Gaussian semigroups
with $\mu_1=\nu_1$, then we have $\mu_t=\nu_t$ for every $t \geq 0$
(see \cite[Section 4]{Pap3}). 
If $(\nu_t)_{t \ge 0}$ is a Gaussian semigroup on  $G$, 
then the corresponding infinitesimal generator $\A$ is of the form
\begin{equation}\label{IG}
\mathcal{A}=\sum_{i, j=1}^{m} 
A_{ij}\va_i\va_j + \sum_{i=1}^{m}b_i \va_i
\end{equation}
for some $m \in \mathbb{N}$, where $\{b_i\}_{i=1}^m \subset \mathbb{R}$ and 
$(A_{ij})_{i, j=1}^{m} \in \mathbb{R}^{m} \otimes \mathbb{R}^{m}$ 
is a symmetric positive semidefinite matrix. 
Conversely, if $\{b_i\}_{i=1}^m \subset \mathbb{R}$ and 
$(A_{ij})_{i, j=1}^{m} \in \mathbb{R}^{m} \otimes \mathbb{R}^{m}$ 
is a symmetric positive semidefinite matrix, then there exists a unique 
convolution semigroup $(\nu_t)_{t \ge 0}$ on $G$ whose infinitesimal 
generator $\A$ coincides with \eqref{IG} (see e.g., \cite{Heyer}).

For a probability measure $\mu$ on $G$, $I \in \mathcal{I}$ and $\alpha>0$, we put
$$
m_\mu^I:=\int_{G} x^I \, \mu(dx), \qquad  
m^\alpha_\mu:=\int_G |x|^\alpha \, \mu(dx). 
$$ 
A probability measure $\mu$ on $G$ is called {\it centered} if 
$m_\mu^I=0$ for $I=[i], \,\, i=1, 2, \dots, d_1. $
A Gaussian semigroup $(\nu_t)_{t \geq 0}$ on $G$ is centered with 
$\nu_t=\tau_{t^{1/2}}\nu_1$ for $t>0$ if and only if
the corresponding infinitesimal generator is of the form
\begin{equation}\label{general generator}
\mathcal{A}=\frac{1}{2}\sum_{i, j=1}^{d_1} 
A_{ij}\va_i\va_j + \sum_{i=d_1+1}^{d_1+d_2}b_i \va_i,
\end{equation}
where $(A_{ij})_{i, j=1}^{d_1} \in \mathbb{R}^{d_1} \otimes \mathbb{R}^{d_1}$ 
is a symmetric positive semidefinite matrix 
and $\{b_i\}_{i=d_1+1}^{d_1+d_2} \subset \mathbb{R}$.  
We note that a Gaussian measure 
$\nu$ has finite moments of arbitrary order. 
As for moments of $\nu$ on $G$, the following is well-known. 
The proof is given for the sake of completeness.

\begin{lm}[cf. {\cite[Lemma 2]{BP}}]\label{Gaussian moment}
Let $(\nu_t)_{t \ge 0}$ be the Gaussian semigroup on $G$ 
whose infinitesimal generator is given by \eqref{general generator}. 
We put $\nu=\nu_1$. Then the following hold. 

\vspace{2mm}
\noindent
{\rm (1)} For a multi-index $I$ satisfying that $d(I)$ is odd, we have
$m_\nu^I=0.$

\vspace{2mm}
\noindent 
{\rm (2)} For a multi-index $I$ with $d(I)=2$, we have
$$
m_\nu^I=
\begin{cases}
b_i & \text{if }I=[i], \, i=d_1+1, d_1+2, \dots, d_1+d_2 \\
A_{ij} & \text{if }I=[i]+[j], \, i, j=1, 2, \dots, d_1
\end{cases}.
$$

\noindent 
{\rm (3)} For  a multi-index $I$ satisfying that $d(I)$ is even and $d(I) \geq 4$, we have
\begin{equation}\label{I-4}
m_\nu^I=\frac{1}{(d(I)/2)!}\sum_{d(J_1)=\cdots=d(J_{d(I)/2})=2}
C_{J_1 \dots J_{d(I)/2}}m^{J_1}_\nu \cdots m^{J_{d(I)/2}}_\nu.
\end{equation}
\end{lm}

\begin{proof}
An identity $\nu=(\tau_{1/\sqrt{2}}(\nu))^2$ and \eqref{CBH} yield
$$
\begin{aligned}
m^I_\nu&=\iint_{G \times G}(xy)^I \, \tau_{1/\sqrt{2}}\nu(dx)\tau_{1/\sqrt{2}}\nu(dy)\\
&=\frac{1}{2^{d(I)/2}}\iint_{G \times G}(xy)^I \, \nu(dx)\nu(dy)
=\frac{1}{2^{d(I)/2}}\Big(2m^I_\nu + \sum_{\substack{d(J)+d(K)=d(I) \\ d(J), d(K) \ge 1}}
C_{JK}m^J_\nu m^K_\nu\Big)
\end{aligned}
$$
for all $I \in \mathcal{I}$. 
The identity above immediately leads to 
\begin{equation}\label{recursive}
m^I_\nu = \frac{1}{2^{d(I)/2}-2}\sum_{\substack{d(J)+d(K)=d(I) \\ d(J), d(K) \ge 1}}C_{JK}
m_\nu^J m_\nu^K, \qquad I \in \mathcal{I}.
\end{equation}

\noindent
(1) Suppose that $d(I)=1$. Then \eqref{recursive} gives rise to
$m^I_\nu=m^I_\nu/\sqrt{2}$, which means $m^I_\nu=0$. 

\vspace{2mm}
\noindent
(2) If $d(I)$ is odd and $d(J)+d(K)=d(I)$ with $d(J), d(K) \ge 1$, then
either $d(J)$ or $d(K)$ is odd and less than $d(I)$. 
Therefore, by induction, we easily obtain $m^I_\nu=0$. 

\vspace{2mm}
\noindent
(3) If $d(I)=4$, Equation \eqref{recursive} gives
$$
m^I_\nu=\frac{1}{2}\sum_{d(J)=d(K)=2}C_{JK}m^J_\nu m^K_\nu.
$$
Suppose that \eqref{I-4} is true for all multi-indices with $d(I) \leq 2k$. 
Then, for a multi-index $I$ with $d(I)=2k+2$, one has
$$
\begin{aligned}
&\sum_{\substack{d(J)+d(K)=2k+2 \\ d(J), d(K) \ge 1}}C_{JK}
m_\nu^J m_\nu^K\\
&=\sum_{\ell=1}^{k}\sum_{\substack{d(J)=2\ell \\ d(K)=(2k+2)-2\ell}}C_{JK}
\Big( \frac{1}{\ell!}\sum_{d(J_1)=\cdots=d(J_\ell)=2}c_{J_1, \dots, J_\ell}
m^{J_1}_\nu \cdots m^{J_\ell}_\nu\Big)\\
&\hspace{1cm}\times
\Big( \frac{1}{(k-\ell+1)!}\sum_{d(J_1)=\cdots=d(J_{k-\ell+1})=2}C_{J_1, \dots, J_{k-\ell+1}}
m^{J_1}_\nu \cdots m^{J_\ell}_\nu\Big)\\
&=\Big(\sum_{\ell=1}^k \frac{1}{\ell!(k-\ell+1)!}\Big)\Big(\sum_{d(J_1)=\cdots=d(J_{k+1})=2}
C_{J_1, \dots, J_{k+1}}m^{J_1}_\nu \cdots m^{J_{k+1}}_\nu\Big)\\
&=\frac{2^{k+1}-2}{(k+1)!}\sum_{d(J_1)=\cdots=d(J_{k+1})=2}
C_{J_1, \dots, J_{k+1}}m^{J_1}_\nu \cdots m^{J_{k+1}}_\nu
\end{aligned}
$$
by using Item (2) and \eqref{CBH} repeatedly. We combine the identity above 
with \eqref{recursive}. Then
we also obtain \eqref{I-4} in the case where $d(I)=2k+2$. This completes the proof. 
\end{proof}

For any $\alpha \ge 0$ and 
any probability measure $\mu$ on $G$, we put
$$
\Lambda_\alpha(\mu):=n^{(2-\alpha)/2}\int_{|x| \ge n^{1/2}}|x|^\alpha \, \mu(dx), \qquad
L_\alpha(\mu):=n^{(2-\alpha)/2}\int_{|x| < n^{1/2}}|x|^\alpha \, \mu(dx). 
$$
We note that, if $m_\mu^{k+2}<\infty$ and 
$\alpha \le k+2$ for some $k \in \mathbb{N}$, then it holds that
\begin{equation}\label{truncation}
\Lambda_\alpha(\mu) \le n^{-k/2}m_\mu^{k+2}, \qquad 
L_\alpha(\mu) \le n^{-(\alpha-2)/2}m_\mu^{\alpha}.
\end{equation}
The estimates of these truncated moments of measures on $G$ play
a crucial role in the proof of Theorem \ref{main}.

\section{{\bf Random walks on covering graphs and modified harmonic realizations}}

Let $\Gamma$ be a torsion free, finitely generated nilpotent group and 
$X=(V, E)$ a $\Gamma$-nilpotent covering graph of a finite graph $X_0=(V_0, E_0)$.  
We define the set of paths in $X$ starting at $x \in V$  by
$$
\Omega_{x, n}(X):=\big\{
c=(e_1, e_2, \dots, e_n) \, \big| \, o(e_{i+1})=t(e_i), \, i=1, 2, \dots, n-1\big\}, 
\qquad n \in \mathbb{N} \cup \{\infty\}. 
$$

We give a transition probability $p : E_0 \to [0, 1]$  satisfying
$\sum_{e \in (E_0)_x}p(e)=1$ for $x \in V_0$ and $p(e)+p(\ol{e})>0$ for $e \in E_0$.
Then,  $p$ yields 
an $X_0$-valued random walk
$(\Omega_x(X_0), \mathbb{P}_x, \{w_n\}_{n=0}^\infty)$ starting at $x \in V_0$, 
where $\mathbb{P}_x$ is 
the probability measure on $\Omega_x(X_0)$ induced from  $p$
and $w_n(c):=o( e_{n+1})$ for $n \in \mathbb{N} \cup\{0\}$
and $c=(e_1, e_2, \dots, e_n, \dots) \in \Omega_x(X_0)$. 
Similarly, we denote by $\mathbb{P}_{x, n}$ for 
the projection of $\mathbb{P}_x$ onto $\Omega_{x, n}(X_0)$. 
In what follows, we assume that the random walk $\{w_n\}_{n=0}^\infty$
associated with $p$ is {\it irreducible}. Then,  
there exists the unique {\it normalized invariant measure} $m : V_0 \to (0, 1]$ which satisfies
$$
\sum_{x \in V_0}m(x)=1, \qquad \sum_{e \in (E_0)_x}p(\ol{e})m\big(t(e)\big), \qquad x \in V_0,
$$
by applying the Perron--Frobenius theorem. 
We put $\widetilde{m}(e):=p(e)m\big(o(e)\big)$ for $e \in E_0$.
Then the symmetry and the  non-symmetry of a random walk is defined as follows:

\begin{df}
A random walk is said to be ($m$-){\it symmetric} if it satisfies $\widetilde{m}(e)=\widetilde{m}(\ol{e})$ for $e \in E_0$. 
Otherwise, it is said to be ($m$-){non-symmetric}. 
\end{df}
A random walk on $X$ is given by a $\Gamma$-invariant lift of the 
random walk on $X_0$. Namely, the transition probability, say also $p : E \to [0, 1]$,
satisfies $p(\gamma e)=p(e)$ for $\gamma \in \Gamma$ and $e \in E$. 
We also write $\mathbb{P}_x$ and $\mathbb{P}_{x, n}$ 
the probability measures on $\Omega_x(X)$ and $\Omega_{x, n}(X)$ induced by $p : E \to [0, 1]$, respectively.
If $c=(e_1, e_2, \dots, e_n) \in \Omega_{x, n}(X)$, then we put $p(c)=p(e_1)p(e_2) \cdots p(e_n)$. 

It is known that the boundary map 
$$
\partial : C_1(X_0, \mathbb{R})=\Big\{\sum_{e \in E_0}a_e e \, \Big| \, a_e \in \mathbb{R}, \, \ol{e}=-e\Big\} \to C_0(X_0, \mathbb{R})=\Big\{\sum_{x \in V_0}a_x x \, \Big| \, a_x \in \mathbb{R}\Big\} 
$$  
between two chain groups of $X_0$ is defined by the linear map 
$\partial(e):=t(e)-o(e)$ for $e \in E_0$. 
Then the {\it first homology group} of $X_0$ is defined by 
$\h_1(X_0, \mathbb{R}):=\mathrm{Ker}(\partial)$. 
We now introduce the {\it homological direction} of $X_0$ defined by
$$
\gamma_p:=\sum_{e \in E_0}\widetilde{m}(e)e \in \h_1(X_0, \mathbb{R}),
$$
which indicates the homological drift of the random walk on $X_0$.  
The following is easily verified. 

\begin{lm}[cf.\,{\cite[Section 2]{KS06}}]
A random walk on $X_0$ is ($m$-)symmetric if and only if $\gamma_p=0$.
\end{lm} 

Let $\rho_{\mathbb{R}} : \h_1(X_0, \mathbb{R}) \to \Gamma/[\Gamma, \Gamma] \otimes \mathbb{R} \cong \g^{(1)}$ be
the canonical surjective linear map
induced by the canonical surjective homomorphism $\rho : \pi_1(X_0) \to \Gamma$, 
where $\pi_1(X_0)$ is the fundamental group of $X_0$. 
A piecewise smooth map 
$\Phi : V \to G$ is called a {\it $\Gamma$-equivariant realization} if 
$\Phi(\gamma x)=\gamma \bullet \Phi(x)$ for $\gamma \in \Gamma$
and $x \in V$. Especially, the notion of the {\it modified harmonic realization}
is introduced in \cite{IKN}, which 
describes the most natural realization of the nilpotent covering graph 
$X$ in a geometric point of view.

\begin{df}[modified harmonic realization, cf.\,\cite{IKN}]
A $\Gamma$-equivariant realization $\Phi_0 : V \to G$ is said to be 
modified harmonic if
\begin{equation}\label{modified harmonic}
\sum_{e \in E_x}p(e)\log\Big(d\Phi_0(e)\Big)\Big|_{\g^{(1)}}=\rho_{\mathbb{R}}(\gamma_p), 
\qquad x \in V,
\end{equation}
where $d\Phi_0(e):=\Phi_0\big(o(e)\big)^{-1}\bullet \Phi_0\big(t(e)\big)$ 
for $e \in E$ and $(\cdot)|_{\g^{(1)}}$ is the projection onto   $\g^{(1)}$. 
\end{df}
Note that such $\Phi_0$ is uniquely determined up to $\g^{(1)}$-translation. 
The quantity appearing on the right-hand side of \eqref{modified harmonic}
is called the {\it asymptotic direction} of the random walk $\{w_n\}_{n=0}^\infty$. 
It is in fact regarded as a mean of the projection of the $G$-valued 
random walk $\xi_n:=\Phi_0(w_n), \, n=0, 1, 2, \dots$, to components 
corresponding to $\g^{(1)}$. See \cite[Section 3]{IKN} for details. 

\begin{lm}[properties on $\rho_{\mathbb{R}}(\gamma_p)$]
{\rm (1)} A law of large numbers holds for the projected random walk
$\{\log(\xi_n)|_{\g^{(1)}}\}_{n=0}^\infty$.  
Namely, we have 
$$
\lim_{n \to \infty}\frac{1}{n}\log(\xi_n)|_{\g^{(1)}} = \rho_{\mathbb{R}}(\gamma_p), \quad \mathbb{P}_x\text{-a.s.}
$$

\vspace{2mm}
\noindent
{\rm (2)} If the random walk $\{w_n\}_{n=0}^\infty$ is ($m$-)symmetric, that is, $\gamma_p=0$, then one has $\rho_{\mathbb{R}}(\gamma_p)=0$. However, the converse 
does not hold in general. 
\end{lm}

The random walk $\{w_n\}_{n=0}^\infty$ is said to be 
{\it centered} if $\rho_{\mathbb{R}}(\gamma_p)=0$. The lemma above 
asserts that a symmetric random walk always satisfies {\bf (A2)}. 
However, there exists a non-symmetric and centered random walk on $X$. 
Thus, it is meaningful to impose {\bf (A2)} in the present paper. 

\section{{\bf Ergodic theorems for transition operators}}

Let us put $\ell^2(X_0):=\{f : V_0 \to \mathbb{C}\}$, which is equipped with 
$$
\la f, g \ra_{\ell^2(X_0)}:=\sum_{x \in V_0}f(x)\ol{g(x)}, \quad \|f\|_{\ell^2(X_0)}:=\Big(\sum_{x \in V_0}|f(x)|^2\Big)^{1/2}, \qquad f, g \in \ell^2(X_0).
$$
We denote by $K_0, \,\, K_0=1, 2, \dots, \leq |V_0|$, for the period 
of the given random walk on $X_0$. Put
$\alpha_k:=e^{2\pi k{\rm i}/K_0}$ for $k=0, 1, \dots, K_0-1$. 
Then the Perron--Frobenius theorem implies that 
the  transition operator $\mathcal{L} : \ell^2(X_0) \to \ell^2(X_0)$ 
has the maximal simple eigenvalues $\alpha_0, \alpha_1, \dots, \alpha_{K_0-1}$
with the corresponding normalized right eigenfunctions 
$\phi_0, \phi_1, \dots, \phi_{K_0-1}$ and left eigenfunctions
$\psi_0, \psi_1, \dots, \psi_{K_0-1}$. 
In particular, we see that $\phi_0(x) \equiv  |V_0|^{-1/2}$ and $\psi(x)=|V_0|^{1/2}m(x)$ for $x \in V_0$. 
We put 
$$
\ell^2_{K_0}(X_0):=\{f \in \ell^2(X_0) \, | \, 
\la f, \psi_j\ra_{\ell^2(X_0)}=0 \text{ for }j=0, 1, \dots, K_0-1\}.
$$
We note that $\LL$ preserves the set $\ell^2_{K_0}(X_0)$, that is, $f \in \ell^2_{K_0}(X_0)$
implies that $\LL f \in \ell^2_{K_0}(X_0)$. 
Therefore, every function $f \in \ell^2(X_0)$ can be decomposed as 
\begin{align}
f=\la f, m \ra_{\ell^2(X_0)}+\sum_{j=1}^{K_0-1}\la f, \psi_j\ra_{\ell^2(X_0)}\phi_j + f_{\ell_{K_0}^2(X_0)}, \qquad f_{\ell_{K_0}^2(X_0)} \in \ell_{K_0}^2(X_0). \label{decomposition}
\end{align}

The following is  the fundamental ergodic theorems for the transition operator 
$\LL$. We refer to   \cite[Theorem 3.2]{IKK} for the proof. 

\begin{pr}[ergodic theorem I]
\label{ergodic} 
Let $\mathcal{L}$ be the transition operator acting on $\ell^2(X_0)$. 
Then we have 
$$
\frac{1}{n}\sum_{k=0}^{n-1}\mathcal{L}^kf(x)=\sum_{x \in V_0}m(x)f(x)+
\frac{1}{n}A[f]_n(x)
$$
for $n \in \mathbb{N}$, $f \in \ell^2(X_0)$ and $x \in V_0$, where 
\begin{equation}\label{decomp}
A[f]_n(x):=\sum_{j=1}^{K_0-1}\la f, \psi_j\ra_{\ell^2(X_0)}\sum_{k=0}^{n-1}\alpha_j^k\phi_j(x) 
+\sum_{k=0}^{n-1}\LL^k f_{\ell_{K_0}^2(X_0)}(x), \qquad n \in \mathbb{N}, \, x \in V_0
\end{equation}
and it satisfies that 
$\|A[f]_n\|_{\ell^2(X_0)} =O(1)$ as $n \to \infty$.
\end{pr}

\vspace{2mm}
As an extension of Proposition \ref{ergodic}, we easily obtain the following. 

\begin{pr}[ergodic theorem II]
\label{ergodic-iterate}
Let $N \in \mathbb{N}$ and $f_1, f_2, \dots, f_N \in \ell^2(X_0)$. 
Then we  have 
\begin{align}\label{ergodic2}
&\frac{1}{n^N}\sum_{\ell_N=0}^{n-N}\cdots\sum_{\ell_2=0}^{\ell_3}
\sum_{\ell_1=0}^{\ell_2}
\mathcal{L}^{\ell_1} f_1(x)\mathcal{L}^{\ell_2+1}f_2(x)\cdots
\LL^{\ell_N+(N-1)}f_N(x)\nonumber\\
&=\frac{1}{N!}\prod_{\ell=1}^N\Big(\sum_{x \in V_0}m(x)f_\ell(x)\Big)+
\sum_{\ell=1}^N
\frac{1}{n^\ell}A[f_1, f_2, \dots, f_N]_n^{(\ell)}(x)
\end{align}
for $x \in V_0$ and sufficiently large $n \in  \mathbb{N}$, where 
the function $A[f_1, f_2, \dots, f_N]_n^{(\ell)}, \, \ell=1, 2, \dots, N$,
satisfies that 
$\|A[f_1, f_2, \dots, f_N]_n^{(\ell)}\|_{\ell^2(X_0)} =O(1)$ as $n \to \infty$. 

\end{pr}

We only give a proof of Proposition \ref{ergodic-iterate}
in the case $N=2$ for the readers' convenience in Appendix, 
because the proof in the case $N \ge 3$ is cumbersome but 
same as the case $N=2$. 

\section{{\bf Validity of the Edgeworth expansion;  a rough observation}}

In this section, we give explicit calculations of all centralized moments of the scaled 
random walks $\{\tau_{n^{-1/2}}(\xi_n)\}_{n=1}^\infty$, where a moment of 
the random walk means the expectation
$$
\mathcal{E}^{x, n, I}_{\Phi_0}:=\mathbb{E}^{x, n}\Big[ \Big(\tau_{n^{-1/2}}\big(\Phi_0(x)^{-1} \bullet \xi_n\big)\Big)^I\Big], 
\qquad x \in V, \, n \in \mathbb{N}, \, I \in \mathcal{I}.
$$
Since the explicit representations of such moments have not been given yet in any references, 
it is worthwhile obtaining them here. 
We emphasize that an extension of the ergodic theorem for the transition operator $\mathcal{L}$ (Proposition \ref{ergodic-iterate}) 
plays a crucial role in the calculations. 
Moreover, we roughly observe the validity of the Edgeworth expansion of the
scaled  random walk by applying the representations of moments, 
which is not rigorous but helpful argument for the readers. 

For a multi-index $I \in \mathcal{I}$, we introduce a function
$F^I_{\Phi_0} : V \to \mathbb{R}$ by
$$
F^I_{\Phi_0}(x):=\sum_{e \in E_x}p(e)\Big(d\Phi_0(e)\Big)^I, \qquad x \in V.
$$
where we recall $d\Phi_0(e):=\Phi_0\big(o(e)\big)^{-1}\bullet \Phi_0\big(t(e)\big)$ 
for $e \in E$. 
Since the function $F^I_{\Phi_0}$ is $\Gamma$-invariant in the sense that 
$F^I_{\Phi_0}(\gamma x)=F^I_{\Phi_0}(x)$ for $\gamma \in \Gamma$ and $x \in V$, 
it can be regarded as a function defined on the base graph $V_0$. 
We also write it for the same symbol $F^I_{\Phi_0} : V_0 \to \mathbb{R}$. 
We put 
$$
m_{\Phi_0}^I:=\sum_{x \in V_0}m(x)F^I_{\Phi_0}(x)=\sum_{e \in E_0}\widetilde{m}(e)
\Big(d\Phi_0(\widetilde{e})\Big)^I, \qquad I \in \mathcal{I}. 
$$
By definition, it is easily seen that 
$$
m_{\Phi_0}^I=\begin{cases}
0 & \text{if }d(I)=1\\
\beta_i(\Phi_0)  & \text{if }I=[i], \, i=d_1+1, d_1+2, \dots, d_1+d_2\\
\sigma_{i}(\Phi_0)\sigma_j(\Phi_0) & \text{if }I=[i]+[j], \, i, j=1, 2, \dots, d_1
\end{cases}.
$$

\subsection{Moments of random walks}
We begin with the cases of low steps, that is,  $d(I) \le 3$. 

\begin{pr}[moments with $d(I) \leq 3$]
\label{phi-moment} 
{\rm (1)} For $I \in \mathcal{I}$ with $d(I)=1$, we have $\mathcal{E}^{x, n, I}_{\Phi_0}=0.$

\vspace{2mm}
\noindent 
{\rm (2)} For $I \in \mathcal{I}$ with $d(I)=2$, we have
$$
\mathcal{E}^{x, n, I}_{\Phi_0}=m^I_{\Phi_0}+\frac{1}{n}A[F^I_{\Phi_0}]_n(x),
$$
where   $A[F^I_{\Phi_0}]_n(x)$ is a function defined by \eqref{decomp}. 

\vspace{2mm}
\noindent 
{\rm (3)} For $I \in \mathcal{I}$ with $d(I)=3$, we have
$$
\begin{aligned}\label{expectation-lower}
&\mathcal{E}^{x, n, I}_{\Phi_0}=
\frac{1}{n^{1/2}}m^I_{\Phi_0}+\frac{1}{n^{3/2}}A[F^I_{\Phi_0}]_n(x).
\end{aligned}
$$
\end{pr}

\begin{proof} (1) By the modified harmonicity of $\Phi_0$, we easily have
$$
\begin{aligned}
\mathcal{E}^{x, n, I}_{\Phi_0}
&=\frac{1}{n^{1/2}}\sum_{c \in \Omega_{x, n-1}(X)}p(c) \Big(\Phi_0(x)^{-1}\bullet\xi_{n-1}(c)\Big)^{I}\\
&\hspace{1cm}+\frac{1}{n^{1/2}}\sum_{e \in E_{t(c)}}p(e)\log\Big(d\Phi_0(e)\Big)\Big|_{\mathfrak{a}^{I}}=\cdots=0. 
\end{aligned}
$$
\noindent
(2) Consider the case where $I=[i], \, i=d_1+1, d_1+2, \dots, d_1+d_2.$
By using 
the Campbell--Baker--Hausdorff formula \eqref{CBH} 
and the modified harmonicity of $\Phi_0$, we have
$$
\begin{aligned}
\mathcal{E}^{x, n, I}_{\Phi_0}
&=\frac{1}{n}\sum_{c \in \Omega_{x, n-1}(X)}p(c)\sum_{e \in E_{t(c)}}p(e)\Big\{ \Big(\Phi_0(x)^{-1}\bullet\xi_{n-1}(c)\Big)^{[i]}+\Big(d\Phi_0(e)\Big)^{[i]}\\
&\hspace{1cm}+\sum_{d(J)=d(K)=1}c_{JK}\Big(\Phi_0(x)^{-1}\bullet\xi_{n-1}(c)\Big)^{J}\Big(d\Phi_0(e)\Big)^{K}\\
&=\frac{1}{n}\sum_{c \in \Omega_{x, n-1}(X)}p(c)\Big(\Phi_0(x)^{-1}\bullet\xi_{n-1}(c)\Big)^{[i]}+\frac{1}{n}\mathcal{L}^{n-1} F^{[i]}_{\Phi_0}(x)
=\cdots=\frac{1}{n}\sum_{k=0}^{n-1}\mathcal{L}^k F^{[i]}_{\Phi_0}(x). 
\end{aligned}
$$
Since the function $F^{[i]}_{\Phi_0}$ is regarded as a function on $X_0$, 
we can apply Proposition \ref{ergodic} to obtain
$$
\frac{1}{n}\sum_{k=0}^{n-1}\LL^k F^{[i]}_{\Phi_0}(x)
=\sum_{x \in V_0}m(x)F^{[i]}_{\Phi_0}(x)+\frac{1}{n}A[F^{[i]}_{\Phi_0}](x)
=m^{[i]}_{\Phi_0}+\frac{1}{n}A[F^{[i]}_{\Phi_0}](x).
$$
As for the case
$I=[i]+[j], \, i, j=1, 2,. \dots, d_1$, the proof can be done 
in the same way as above. 

\vspace{2mm}
\noindent
(3) Equation \eqref{CBH} yields that  
\begin{align}
\mathcal{E}^{x, n, I}_{\Phi_0}&=\frac{1}{n^{3/2}}\sum_{c \in \Omega_{x, n-1}(X)}p(c)\sum_{e \in E_{t(c)}}p(e)\Big\{
  \Big(\Phi_0(x)^{-1}\bullet\xi_{n-1}(c)\Big)^I+\Big(d\Phi_0(e)\Big)^{I} \nn\\
  &\hspace{1cm}+\sum_{\substack{d(J)+d(K)=3 \\ d(J), d(K) \ge 1}}c_{J, K}
  \Big(\Phi_0(x)^{-1}\bullet\xi_{n-1}(c)\Big)^J\Big(d\Phi_0(e)\Big)^{K} \Big\}.
  \label{I=3;eq1}
\end{align}
By noting that the random variables 
$\big(\Phi_0(x)^{-1}\bullet\xi_{n-1}(c)\big)^J$ and $\big(d\Phi_0(e)\big)^{K}$
are independent, we have
$$
\begin{aligned}
&\sum_{c \in \Omega_{x, n-1}(X)}p(c)\sum_{e \in E_{t(c)}}p(e)
\sum_{\substack{d(J)+d(K)=3 \\ d(J), d(K) \ge 1}}c_{J, K}
  \Big(\Phi_0(x)^{-1}\bullet\xi_{n-1}(c)\Big)^J\Big(d\Phi_0(e)\Big)^{K}\\
 &=\sum_{\substack{d(J)+d(K)=3 \\ d(J), d(K) \ge 1}}c_{J, K}\Big\{\sum_{c \in \Omega_{x, n-1}(X)}p(c)\Big(\Phi_0(x)^{-1}\bullet\xi_{n-1}(c)\Big)^J\Big\}
 \LL^{n-1}F^K_{\Phi_0}(x)=0,
\end{aligned}
$$
where we used Item (1). 
Therefore, we inductively have
\begin{align}
\mathcal{E}^{x, n, I}_{\Phi_0}&=\frac{1}{n^{3/2}}\Big\{\sum_{c \in \Omega_{x, n-1}(X)}p(c)\Big(\Phi_0(x)^{-1}\bullet\xi_{n-1}(c)\Big)^I + 
\LL^{n-1}F_{\Phi_0}^I(x)\Big\}\nn\\
&=\frac{1}{n^{3/2}}\sum_{k=0}^{n-1}\LL^k F_{\Phi_0}^I(x). \label{I=3;eq2}
\end{align}
Then, it follows form Proposition \ref{ergodic} that 
\begin{equation}\label{I=3;eq3}
\frac{1}{n}\sum_{k=0}^{n-1}\LL^k F_{\Phi_0}^I(x)
=m^I_{\Phi_0}+\frac{1}{n}A[F^I_{\Phi_0}]_n(x).
\end{equation}
By combining \eqref{I=3;eq1} with \eqref{I=3;eq2} and \eqref{I=3;eq3}, we obtain
$$
\begin{aligned}
&\mathcal{E}^{x, n, I}_{\Phi_0}
=\frac{1}{n^{1/2}}m^I_{\Phi_0}+\frac{1}{n^{3/2}}A[F^I_{\Phi_0}]_n(x),
\end{aligned}
$$
which completes the proof of Proposition \ref{phi-moment}. 
\end{proof}

\vspace{2mm}
We next discuss the cases of higher steps, that is, $d(I) \geq 4$. 
For $N \in \mathbb{N}$, $f_1, f_2, \dots, f_N \in \ell^2(X_0)$ and
sufficiently large $n \geq N$, we put  
$$
\begin{aligned}
&\mathcal{Q}_n^{(N)}[f_1, f_2, \dots, f_N](x)\\
&:=\sum_{\ell_N=0}^{n-N}\cdots\sum_{\ell_2=0}^{\ell_3}
\sum_{\ell_1=0}^{\ell_2}
\mathcal{L}^{\ell_1} f_1(x)\mathcal{L}^{\ell_2+1}f_2(x)\cdots
\LL^{\ell_N+(N-1)}f_N(x), \qquad x \in V_0.
\end{aligned}
$$

By using the Campbell--Baker--Hausdorff formula \eqref{CBH}, we easily 
obtain the following. The proof is so cumbersome but is straightforward as in 
the previous lemma. 

\begin{lm}\label{pre-moment}
 If a multi-index $I$ satisfies $d(I) \geq 4$, we have
\begin{align}
&\mathbb{E}^{x, n}\Big[\Big(\Phi_0(x)^{-1} \bullet \xi_n\Big)^I\Big]\nn\\
&=\mathcal{Q}_n^{(1)}[F^I_{\Phi_0}](x)
+\sum_{q_1=1}^{[d(I)/2]-1}\Big(\sum_{\substack{d(J_1)=2q_1 \\ d(K_1)=d(I)-2q_1 }}+
\sum_{\substack{d(J_1)=2q_1+1 \\ d(K_1)=d(I)-2q_1-1 }}\Big)c_{J_1K_1}\Big(\mathcal{Q}_n^{(2)}[F_{\Phi_0}^{J_1}F_{\Phi_0}^{K_1}](x)\nn\\
&\hspace{1cm}+\sum_{q_2=1}^{[d(J_1)/2]-1}\Big(\sum_{\substack{d(J_2)=2q_2 \\ d(K_2)=d(J_1)-2q_2 }}+
\sum_{\substack{d(J_2)=2q_2+1 \\ d(K_2)=d(J_1)-2q_1-1 }}\Big)c_{J_2K_2}\Big(
\mathcal{Q}_n^{(3)}[F_{\Phi_0}^{J_2}F_{\Phi_0}^{K_2}F_{\Phi_0}^{K_1}](x)+\cdots\nn\\
&\hspace{1cm}+\sum_{q_{[d(I)/2]-1}=1}^{[d(J_{[d(I)/2]-2})/2]-1}
\Big(\sum_{\substack{d(J_{[d(I)/2]-1})=2q_{[d(I)/2]-1} \\ d(K_{[d(I)/2]-1})
=d(J_{[d(I)/2]-1})-2q_{[d(I)/2]-1} }}\nn\\
&\hspace{1cm}
+\sum_{\substack{d(J_{[d(I)/2]-1})=2q_{[d(I)/2]-1}+1 \\ d(K_{[d(I)/2]-1})
=d(J_{[d(I)/2]-1})-2q_{[d(I)/2]-1}-1 }}
\Big)c_{J_{[d(I)/2]-1}K_{[d(I)/2]-1}}\nn\\
&\hspace{1cm}\times 
\mathcal{Q}_n^{([d(I)/2])}[F_{\Phi_0}^{J_{[d(I)/2]-1}}F_{\Phi_0}^{K_{[d(I)/2]-1}}
F_{\Phi_0}^{K_{[d(I)/2]-2}}\cdots F_{\Phi_0}^{K_{1}}](x)\Big)\cdots\Big)\Big)
\label{bullet-expansion}
\end{align}
for $x \in V_0$ and sufficiently large $n \in \mathbb{N}$, 
where the summation $\sum_{d(J)=a, d(K)=1}$ 
is regarded as zero for $a=3, 4, \dots$ due
to the modified harmonicity of $\Phi_0$. 
\end{lm}

By using  Proposition \ref{ergodic-iterate} and Lemma \ref{pre-moment}, 
we immediately obtain the following.

\begin{tm}[moments with $d(I) \le 4$]
\label{phi-moment-general}
Let $I$ be a multi-index with $d(I) \geq 4$ and $n \in \mathbb{N}$ sufficiently large. 
Then we have
\begin{align}
\mathcal{E}_{\Phi_0}^{x, n, I}
=\frac{a_{\Phi_0}^{x, n, I}([d(I)/2])}{n^{d(I)/2-[d(I)/2]}}
+\cdots+\frac{a_{\Phi_0}^{x, n, I}(1)}{n^{d(I)/2-1}}
+\frac{a_{\Phi_0}^{x, n, I}(0)}{n^{d(I)/2}}, \qquad x \in V_0,
\label{expectation-expansion}
\end{align}
where each coefficient $a_{\Phi_0}^{x, n, I}(i), \, i=0, 1, 2, \dots, [d(I)/2]$ is given by \eqref{a_0}, \eqref{a_1}
and \eqref{a_l} below. Moreover, it holds that 
\begin{equation}\label{a-estimate}
\|a_{\Phi_0}^{x, n, I}(k)\|=O(1) \quad (n \to \infty),\qquad k=0, 1, 2, \dots, [d(I)/2]. 
\end{equation}
\end{tm}

\begin{proof}  It follows from Proposition \ref{ergodic-iterate} 
and Lemma \ref{pre-moment} that 
\begin{align}
\mathcal{E}_{\Phi_0}^{x, n, I}
&=\frac{1}{n^{d(I)/2-1}}\Big(m^I_{\Phi_0}+\frac{1}{n}A[F^I_{\Phi_0}]_n(x)\Big)
+\sum_{q_1=1}^{[d(I)/2]-1}\Big(\sum_{\substack{d(J_1)=2q_1 \\ d(K_1)=d(I)-2q_1 }}+
\sum_{\substack{d(J_1)=2q_1+1 \\ d(K_1)=d(I)-2q_1-1 }}\Big)c_{J_1K_1}\nn\\
&\hspace{1cm}\times\Big(\frac{1}{n^{d(I)/2-2}}
\Big(\frac{1}{2}m^{J_1}_{\Phi_0}m^{K_1}_{\Phi_0}
+\frac{1}{n}A[F^{J_1}_{\Phi_0}, F^{K_1}_{\Phi_0}]_n^{(1)}(x)
+\frac{1}{n^2}A[F^{J_1}_{\Phi_0}, F^{K_1}_{\Phi_0}]_n^{(2)}(x)\Big)+\cdots\nn\\
&\hspace{1cm}+
\sum_{q_{[d(I)/2]-1}=1}^{[d(J_{[d(I)/2]-2})/2]-1}
\Big(\sum_{\substack{d(J_{[d(I)/2]-1})=2q_{[d(I)/2]-1} \\ d(K_{[d(I)/2]-1})=d(J_{[d(I)/2]-1})-2q_{[d(I)/2]-1} }}\nn\\
&\hspace{1cm}+\sum_{\substack{d(J_{[d(I)/2]-1})=2q_{[d(I)/2]-1}+1 \\ d(K_{[d(I)/2]-1})=d(J_{[d(I)/2]-1})-2q_{[d(I)/2]-1}-1 }}\Big)c_{J_{[d(I)/2]-1}K_{[d(I)/2]-1}}\nn\\
&\hspace{1cm}\times  \frac{1}{n^{d(I)/2 - [d(I)/2]}}
\Big(\frac{1}{([d(I)/2])!}
m_{\Phi_0}^{J_{[d(I)/2]-1}}m_{\Phi_0}^{K_{[d(I)/2]-1}}
m_{\Phi_0}^{K_{[d(I)/2]-2}}\cdots m_{\Phi_0}^{K_{1}}\nn\\
&\hspace{1cm}+\sum_{\ell=1}^{[d(I)/2]}\frac{1}{n^\ell}A[F_{\Phi_0}^{J_{[d(I)/2]-1}}, F_{\Phi_0}^{K_{[d(I)/2]-1}}, 
F_{\Phi_0}^{K_{[d(I)/2]-2}}, \cdots, F_{\Phi_0}^{K_{1}}]_n^{(\ell)}(x)\Big)\cdots\Big)\nn\\
&=\frac{a_{\Phi_0}^{x, n, I}([d(I)/2])}{n^{d(I)/2-[d(I)/2]}}
+\cdots+\frac{a_{\Phi_0}^{x, n, I}(1)}{n^{d(I)/2-1}}
+\frac{a_{\Phi_0}^{x, n, I}(0)}{n^{d(I)/2}},\nn
\end{align}
where 
\begin{align}
a_{\Phi_0}^{x, n, I}(0)&=A[F^I_{\Phi_0}]_n(x)+\sum_{q_1}\sum_{J_1, K_1}c_{J_1K_1}\Big(
A[F^{J_1}_{\Phi_0}, F^{K_1}_{\Phi_0}]_n^{(2)}(x)+\cdots\nn\\
&\hspace{1cm}+\sum_{q_{[d(I)/2]-1}}\sum_{J_{[d(I)/2]-1}, K_{[d(I)/2]-1}}
c_{J_{[d(I)/2]-1} K_{[d(I)/2]-1}}\nn\\
&\hspace{1cm}\times
A[F_{\Phi_0}^{J_{[d(I)/2]-1}}, F_{\Phi_0}^{K_{[d(I)/2]-1}}, 
F_{\Phi_0}^{K_{[d(I)/2]-2}}, \cdots, F_{\Phi_0}^{K_{1}}]_n^{([d(I)/2])}(x)\Big)\cdots\Big),
\label{a_0}\\
a_{\Phi_0}^{x, n, I}(1)&=m^I_{\Phi_0}+\sum_{q_1}\sum_{J_1, K_1}c_{J_1K_1}\Big(
A[F^{J_1}_{\Phi_0}, F^{K_1}_{\Phi_0}]_n^{(1)}(x)+\cdots\nn\\
&\hspace{1cm}+\sum_{q_{[d(I)/2]-1}}\sum_{J_{[d(I)/2]-1}, K_{[d(I)/2]-1}}
c_{J_{[d(I)/2]-1} K_{[d(I)/2]-1}}\nn\\
&\hspace{1cm}\times
A[F_{\Phi_0}^{J_{[d(I)/2]-1}}, F_{\Phi_0}^{K_{[d(I)/2]-1}}, 
F_{\Phi_0}^{K_{[d(I)/2]-2}}, \cdots, F_{\Phi_0}^{K_{1}}]_n^{([d(I)/2])}(x)\Big)\cdots\Big),
\label{a_1}
\end{align}
and 
\begin{align}
a_{\Phi_0}^{x, n, I}(\ell)&=\sum_{q_1}\sum_{J_1, K_1}c_{J_1K_1}\Big(\cdots
\sum_{q_{\ell-1}}\sum_{J_{\ell-1}, K_{\ell-1}}c_{J_{\ell-1}K_{\ell-1}}\Big(
\frac{1}{\ell!}m^{J_{\ell-1}}_{\Phi_0}m^{K_{\ell-1}}_{\Phi_0}m^{K_{\ell-2}}_{\Phi_0}\cdots
m^{K_{1}}_{\Phi_0}\nn\\
&\hspace{1cm}+\sum_{q_{\ell}}\sum_{J_{\ell}, K_{\ell}}c_{J_\ell K_\ell}
\Big(A[F_{\Phi_0}^{J_{\ell}}, F_{\Phi_0}^{K_{\ell}}, 
F_{\Phi_0}^{K_{\ell-1}}, \cdots, F_{\Phi_0}^{K_{1}}]_n^{(1)}(x)+\cdots\nn\\
&\hspace{1cm}+\sum_{q_{[d(I)/2]-1}}\sum_{J_{[d(I)/2]-1}, K_{[d(I)/2]-1}}
c_{J_{[d(I)/2]-1} K_{[d(I)/2]-1}}\nn\\
&\hspace{1cm}\times
A[F_{\Phi_0}^{J_{[d(I)/2]-1}}, F_{\Phi_0}^{K_{[d(I)/2]-1}}, 
F_{\Phi_0}^{K_{[d(I)/2]-2}}, \cdots, F_{\Phi_0}^{K_{1}}]_n^{([d(I)/2]-\ell)}(x)\Big)\Big)\cdots\Big)
\label{a_l}
\end{align}
for $\ell=2, 3, \dots, [d(I)/2]$. Equation \eqref{a-estimate} is then clearly obtained 
by definition.  
This completes the proof of Lemma \ref{phi-moment-general}. 
\end{proof}

\subsection{Validity of the Edgeworth expansion via a rough observation }

In this subsection, we intuitively observe that the Edgeworth expansion of the scaled random walk 
$\{\tau_{n^{-1/2}}(\xi_n)\}_{n=1}^\infty$ is valid  
by using Proposition \ref{phi-moment} and Theorem \ref{phi-moment-general}. 
They are roughly obtained by the direct application of the stratified Taylor expansion formula to each expectations
$\LL^n \mathcal{P}^H_{n^{-1/2}}f$ and $\mathcal{P}^H_{n^{-1/2}}\e^{\A(\Phi_0)} f$. 

In what follows, we suppose that the function $f : G \to \mathbb{R}$ is bounded and analytic.  
For the Gaussian semigroup $(\nu^{\Phi_0}_t)_{t \ge 0}$ associated with $\A(\Phi_0)$, 
we write the corresponding Gaussian measure as $\nu=\nu_1^{\Phi_0}$. 
We apply the right stratified Taylor formula to
the function $f$ at 
$\tau_{n^{-1/2}}(\Phi_0(x))$. Then we have the Taylor series expansion
$$
\begin{aligned}
\mathcal{L}^n\mathcal{P}_{n^{-1/2}}^H f(x)
&=\mathbb{E}^{x, n}\Big[f\Big(\tau_{n^{-1/2}}\big(\xi_n)\Big)\Big]\\
&=\mathbb{E}^{x, n}\Big[f\Big(\tau_{n^{-1/2}}\Big(\Phi_0(x) \bullet \big(\Phi_0(x)^{-1}\bullet \xi_n\big)\Big)\Big)\Big]\\
&=\sum_{d(I)=1}^\infty \widehat{S}^I f\Big(\tau_{n^{-1/2}}\big(\Phi_0(x)\big)\Big)\mathcal{E}^{x, n, I}_{\Phi_0}, \qquad x \in V, \, n \in \mathbb{N}. 
\end{aligned}
$$
On the other hand, it also follows from the right stratified Taylor formula that
$$
\begin{aligned}
\mathcal{P}^H_{n^{-1/2}}\e^{\A(\Phi_0)}f(x)
&=\mathbb{E}^\nu\Big[ f\Big(\tau_{n^{-1/2}}(\Phi_0(x)) \bullet \mathcal{X}_1\big)\Big)\Big]\\
&=\sum_{d(I)=1}^\infty \widehat{S}^I f\Big(\tau_{n^{-1/2}}\big(\Phi_0(x)\big)\Big)
\mathbb{E}^\nu[(\mathcal{X}_1)^I]\\
&=\sum_{d(I)=1}^\infty \widehat{S}^I f\Big(\tau_{n^{-1/2}}\big(\Phi_0(x)\big)\Big)
m_\nu^I, \qquad x \in V, \, n \in \mathbb{N},
\end{aligned}
$$
where $(\mathcal{X}_t)_{t \ge 0}$ denotes the $G$-valued diffusion process whose 
infinitesimal generator is given by $\A$. 
Then, for $x \in V$ and $n \in \mathbb{N}$, we have
\begin{align}
&\mathcal{L}^n\mathcal{P}_{n^{-1/2}}^H f(x) - 
\mathcal{P}^H_{n^{-1/2}}\e^{\A(\Phi_0)}f(x)
=\sum_{d(I)=1}^\infty \widehat{S}^I f\Big(\tau_{n^{-1/2}}\big(\Phi_0(x)\big)\Big)
(\mathcal{E}^{x, n, I}_{\Phi_0} - m_\nu^I).
\label{dif-expansion}
\end{align}

Let us consider the difference 
$\mathcal{Z}^{x, n, I}:=\mathcal{E}^{x, n, I}_{\Phi_0} - m_\nu^I$ for
$x \in V, \, n \in \mathbb{N}$ and $I \in \mathcal{I}$, 
appearing on the right-hand side of \eqref{dif-expansion}. 
At first, suppose that $d(I)=1$. Then Lemma \ref{phi-moment}-(1) 
and $m^I_\nu=0$
immediately imply that $\mathcal{Z}^{x, n, I}=0$. 
We next consider the case where $d(I)=2$. 
By applying Lemma \ref{phi-moment}-(2) and an identity 
$m^I_{\Phi_0}=m^I_\nu$ for $d(I)=2$, 
we have
$$
\mathcal{Z}^{x, n, I}=\frac{1}{n}A[F^I_{\Phi_0}]_n(x). 
$$
If the multi-index $I$ satisfies $d(I)=3$, it follows from Lemma \ref{phi-moment}-(3) and $m^I_\nu=0$ that
$$
\mathcal{Z}^{x, n, I}=\frac{1}{n^{1/2}}m^I_{\Phi_0}+\frac{1}{n^{3/2}}A[F^I_{\Phi_0}]_n(x).
$$
Next suppose that $d(I)$ is odd and $d(I) \ge 5$. 
By using Lemma \ref{phi-moment-general} and $m^I_\nu=0$, we have
$$
\mathcal{Z}^{x, n, I}=\frac{a_{\Phi_0}^{x, n, I}\big((d(I)-1)/2\big)}{n^{1/2}}
+\cdots+\frac{a_{\Phi_0}^{x, n, I}(1)}{n^{d(I)/2-1}}+\frac{a_{\Phi_0}^{x, n, I}(0)}{n^{d(I)/2}}.
$$
The rest is the case where $d(I)$ is even and $d(I) \ge 4$. 
The again use of Lemma \ref{phi-moment-general} implies that
$$
\mathcal{Z}^{x, n, I}=a_{\Phi_0}^{x, n, I}(d(I)/2)+\frac{a_{\Phi_0}^{x, n, I}(d(I)/2-1)}{n}+\cdots+
\frac{a_{\Phi_0}^{x, n, I}(1)}{n^{d(I)/2-1}}+\frac{a_{\Phi_0}^{x, n, I}(0)}{n^{d(I)/2}}-m^I_\nu.
$$
By noting $m^I_{\Phi_0}=m^I_\nu$ for a multi-index $I$ with $d(I)=2$, we have
$$
\begin{aligned}
a_{\Phi_0}^{x, n, I}(d(I)/2)&=\sum_{\substack{d(J_1)=d(I)-2 \\ d(K_1)=2}}C_{J_1K_1}
\sum_{\substack{d(J_2)=d(I)-4 \\ d(K_2)=2}}C_{J_2K_2}
\cdots\sum_{\substack{d(J_{d(I)/2-1})=2 \\ 
d(K_{d(I)/2-1})=2}}C_{J_{d(I)/2-1}K_{d(I)/2-1}}\\
&\hspace{1cm}\times \frac{1}{(d(I)/2)!}m_{J_{d(I)/2-1}}^{\Phi_0}
m^{K_{d(I)/2-1}}_{\Phi_0}m^{K_{d(I)/2-2}}_{\Phi_0}\cdots m^{K_{1}}_{\Phi_0}\\
&=\frac{1}{(d(I)/2)!}\sum_{d(J_1)=\cdots=d(J_{d(I)/2})=2}C_{J_1\dots J_{d(I)/2}}
m^{J_1}_{\nu}\cdots m^{J_{d(I)/2}}_{\nu}=m^I_\nu,
\end{aligned}
$$
where we used \eqref{I-4} for the final line. Therefore, we have
$$
\mathcal{Z}^{x, n, I}=\frac{a_{\Phi_0}^{x, n, I}(d(I)/2-1)}{n}+\cdots+
\frac{a_{\Phi_0}^{x, n, I}(1)}{n^{d(I)/2-1}}+\frac{a_{\Phi_0}^{x, n, I}(0)}{n^{d(I)/2}}.
$$
By combining \eqref{dif-expansion} with the calculations of 
$\mathcal{Z}^{x, n, I}$ above, 
we obtain
\begin{align}
\mathcal{L}^n\mathcal{P}_{n^{-1/2}}^H f(x) - 
\mathcal{P}^H_{n^{-1/2}}\e^{\A}f(x)
&=\sum_{d(I)=1}^\infty \widehat{S}^I f\Big(\tau_{n^{1/2}}\big(\Phi_0(x)\big)\Big)
\mathcal{Z}^{x, n, I}=\sum_{j=1}^\infty \frac{c_j}{n^{j/2}},
\label{pre-exp}
\end{align}
where each coefficient $c_j^n=c_j^n(x, f, p, \Phi_0), \, j=1, 2, \dots$, is given in the following way. 
We put 
$\mathcal{S}_k:=\{k, \, k+2, \, k+4, \dots\}$ for $k=1, 2, \dots$.  
Then one has
\begin{align}
c_1^n &= \sum_{d(I)=3}\widehat{S}^I f\Big( \tau_{n^{-1/2}}\big(\Phi_0(x)\big)\Big)
m^I_{\Phi_0}\nn\\
&\hspace{1cm}+\sum_{s \in \mathcal{S}_5}\sum_{d(I)=s} 
\widehat{S}^{I}f\Big( \tau_{n^{-1/2}}\big(\Phi_0(x)\big)\Big)
a_{\Phi_0}^{x, n, I}\Big(\frac{s-1}{2}\Big), \\
c_2^n &= \sum_{d(I)=2} \widehat{S}^{I}f\Big( \tau_{n^{-1/2}}\big(\Phi_0(x)\big)\Big)
A[F^I_{\Phi_0}]_n(x)\nn\\
&\hspace{1cm}+\sum_{s \in \mathcal{S}_4}\sum_{d(I)=s} 
\widehat{S}^{I}f\Big( \tau_{n^{-1/2}}\big(\Phi_0(x)\big)\Big)
a_{\Phi_0}^{x, n, I}\Big(\frac{s-2}{2}\Big),\\
c_3^n &= \sum_{d(I)=3} 
\widehat{S}^{I}f\Big( \tau_{n^{-1/2}}\big(\Phi_0(x)\big)\Big)
A[F^I_{\Phi_0}]_n(x)\nn\\
&\hspace{1cm}+\sum_{s \in \mathcal{S}_5}\sum_{d(I)=s}
\widehat{S}^{I}f\Big( \tau_{n^{-1/2}}\big(\Phi_0(x)\big)\Big)
a_{\Phi_0}^{x, n, I}\Big(\frac{s-3}{2}\Big),
\end{align}
and 
\begin{align}
c_j^n &= \sum_{s \in \mathcal{S}_j}\sum_{d(I)=s} 
\widehat{S}^{I}f\Big( \tau_{n^{-1/2}}\big(\Phi_0(x)\big)\Big)
a_{\Phi_0}^{x, n, I}\Big(\frac{s-k}{2}\Big), \qquad j=4, 5, \dots. 
\end{align}
Though each coefficient $c_j^n=c_j^n(x, f, p, \Phi_0), \, j=1, 2, 3, \dots$, 
depends on the choice of $n \in \mathbb{N}$, we can still expect that 
each $c^n_j$ satisfies that $\|c^n_j\|=O(1)$ as $n \to \infty$
for $j=1, 2, 3, \dots$, by applying \eqref{a-estimate}. 
 Thus,  
Equation \eqref{pre-exp} should imply the possibility of the validity of 
the Edgeworth expansion of the scaled random 
walk $\{\tau_{n^{-1/2}}(\xi_n)\}_{n=1}^\infty$. 
However, the argument in this subsection does not imply
the rigorous proof of Theorem \ref{main}. Indeed, 
several technical difficulties appear when we deal with  the remainder terms 
of the Taylor expansions of both $\mathcal{L}^n\mathcal{P}_{n-^{-1/2}}^H f$
and $\mathcal{P}_{n^{-1/2}}^H \e^{\A(\Phi_0)} f$. 
Therefore, we need another technical approach for the proof, while the observation above must be important.

\section{{\bf Proof of Theorem \ref{main}}}

The aim of this section is to give a proof of Theorem \ref{main}, 
the main result of the present paper.

\subsection{Proof of Theorem \ref{main}}
Let $x \in V$ and $\{W_i\}_{i=1}^\infty$ be a sequence of random variables on $\Omega_x(X)$ 
with values in $G$ 
defined by
$W_i(c):=d\Phi_0(e_i)$
for $c=(e_1, e_2, \dots, e_i, \dots) \in \Omega_x(X)$. 
Note that $\{W_i\}_{i=1}^\infty$ is independent but
 not identically distributed in general by definition and satisfies that 
 $$
 w_n=x \bullet W_1 \bullet W_2 \bullet \cdots \bullet W_n, \qquad n=0, 1, 2, \dots.
 $$
We put 
$\mathbb{P}_m:=\sum_{x \in V_0}\mathbb{P}_xm(x)$ and 
$\mu^{(i)}:=\mathbb{P}_m \circ (W_i)^{-1}$, $i=1, 2, \dots, $
which is the image probability measure of $\mathbb{P}_m$. 
Then we know that the law of each $w_n$, $n=1, 2, \dots$, is written by
the convolution power
$\mu^{(1)} * \mu^{(2)} * \cdots * \mu^{(n)}$. Moreover, we observe 
$$
m^I_{\mu^{(k)}}=\int_G x^I \, \mu^{(k)}(dx)=\sum_{e \in E_0}\widetilde{m}(e)
\Big(d\Phi_0(e)\Big)^I=m^I_{\Phi_0}, \qquad I \in \mathcal{I}, \, k=1, 2, \dots,
$$
and $m_{\mu^{(k)}}^N<\infty$ for $k, N=1, 2, \dots$. 
We now give a proof of Theorem \ref{main} only in the case of $t=1$,
since general cases are shown similarly to the case of $t=1$.

\begin{proof}[Proof of Theorem \ref{main}]

We basically follows the argument given by \cite{BP} and \cite{Pap2},
although we need a careful examination to follow it.  Indeed,
random walks discussed in both \cite{BP} and \cite{Pap2} are independently and 
identically distributed, while our random walks are not always so. 
Moreover, there is a geometric constraint in terms of 
the modified harmonicity of the realization $\Phi_0$  in our setting. 
The basic idea for the proof is to make use of the identity
\begin{align}
&\tau_{n^{-1/2}}(\mu^{(1)} * \mu^{(2)} * \cdots * \mu^{(n)}) \nn\\
&=\nu+\sum_{k=1}^n \tau_{n^{-1/2}}\nu^{k-1}*(\tau_{n^{-1/2}}\mu^{(k)} - \tau_{n^{-1/2}}\nu)*\tau_{n^{-1/2}}(\mu^{(k+1)}*\cdots*\mu^{(n)})
\label{Proof; measure-decomp}
\end{align}
for $n=1, 2, \dots$, where we again recall that $\nu=\nu_1^{\Phi_0}$ is the centered Gaussian measure corresponding 
to $(\nu_t^{\Phi_0})_{t \ge 0}$, satisfying 
$\tau_{n^{-1/2}}\nu^{*n}=\nu$. 
 In what follows, we always omit the convolution symbol $*$ for the simplicity of notations. 
By using \eqref{Proof; measure-decomp}, we have
\begin{align}
&\mathcal{L}^n\mathcal{P}_{n^{-1/2}}^Hf(x) 
- \mathcal{P}^H_{n^{-1/2}}\e^{\mathcal{A}(\Phi_0)}f(x) \nn\\
&=\sum_{k=1}^n \iiint_{G^3} f(x \bullet x_1 \bullet y_1 \bullet x_2) 
 \tau_{n^{-1/2}}\nu^{k-1}(dx_1)(\tau_{n^{-1/2}}\mu^{(k)} - \tau_{n^{-1/2}}\nu)(dy_1)\nn\\
 &\hspace{1cm}\times\tau_{n^{-1/2}}(\mu^{(k+1)}\cdots\mu^{(n)})(dx_2), 
\qquad f \in C_b^{3(N+1)r}(G), \, x \in V. 
\end{align}
We put 
$$
\begin{aligned}
\mathcal{I}_1&:=\sum_{k=1}^n\iiint_{|y_1| \ge 1} 
f(x \bullet x_1 \bullet y_1 \bullet x_2) \\
&\hspace{1cm}\times
\tau_{n^{-1/2}}\nu^{k-1}(dx_1)(\tau_{n^{-1/2}}\mu^{(k)} - \tau_{n^{-1/2}}\nu)(dy_1)\tau_{n^{-1/2}}(\mu^{(k+1)}\cdots\mu^{(n)})(dx_2), \\
\mathcal{I}_2&:=\sum_{k=1}^n\iiint_{|y_1| < 1} 
f(x \bullet x_1 \bullet y_1 \bullet x_2) \\
&\hspace{1cm}\times
\tau_{n^{-1/2}}\nu^{k-1}(dx_1)(\tau_{n^{-1/2}}\mu^{(k)} - \tau_{n^{-1/2}}\nu)(dy_1)\tau_{n^{-1/2}}(\mu^{(k+1)}\cdots\mu^{(n)})(dx_2).
\end{aligned}
$$

We split the proof into seven steps. 

\vspace{2mm}
\noindent
{\bf Step 1.} It follows from \eqref{truncation} that 
$$
\begin{aligned}
|\mathcal{I}_1| &\le \frac{1}{n}\|f\|_\infty \sum_{k=1}^n \Big( \int_{|y_1| \ge 1} \tau_{n^{-1/2}}\mu^{(k)}(dy_1)
+ \int_{|y_1| \ge 1}\tau_{n^{-1/2}}\nu(dy_1)\Big)\\
&= \frac{1}{n^2}\|f\|_\infty  \sum_{k=1}^n\big(\Lambda_0(\mu^{(k)})+\Lambda_0(\nu)\big) \\
&\le \frac{1}{n^2}\|f\|_\infty n^{-N/2}
\sum_{k=1}^n (m_{\mu^{(k)}}^{N+2}+m_\nu^{N+2})=O(n^{-N/2}),
\end{aligned}
$$
which means that the contribution of  $\mathcal{I}_1$ is nothing but up to $O(n^{-N/2})$.

\vspace{2mm}
\noindent
{\bf Step 2.} At the rest steps, we concentrate on looking into the integral $\mathcal{I}_2$. On the set $\{|y_1| <1\}$, 
we start with applying the right stratified Tayler formula (Lemma \ref{Taylor})
up to the order $N+1$ to the function $f(x \bullet x_1 \bullet y_1 \bullet x_2)$
at $\bm{1}_G$ with respect to $y_1$. 
Then, for $x_1, x_2 \in G$, we have 
\begin{equation}\label{Proof; Taylor1}
f(x \bullet x_1 \bullet y_1 \bullet x_2)=\sum_{d(I_1) \le N+1} 
\widehat{S}^{I_1}_{(2)}f(x \bullet x_1 \bullet \bm{1}_G \bullet x_2)
y_1^{I_1} + R^f_{N+2}(x_1, y_1, x_2)
\end{equation}
with the estimate 
\begin{align}
&| R^f_{N+2}(x_1, y_1, x_2)| \nn\\
&\le C(G)|y_1|^{N+2}
\sup\big\{ |\widehat{\mathfrak{a}}^{I_1}f(x \bullet x_1 \bullet y_1' \bullet x_2)| \, 
: \, d(I_1)=N+2, \, |y_1'| \le b(G)^{N+2}|y_1|\big\}. \nn
\end{align}
By applying Lemma \ref{estimate of f}, one can give the estimation
\begin{equation}\label{Proof; Remainder1}
| R^f_{N+2}(x_1, y_1, x_2)| \le C(G)|y_1|^{N+2}
(1+|x \bullet x_1|^{(N+2)(r-1)})\|D^{(N+2)r}f\|_\infty.
\end{equation}
Therefore, \eqref{truncation} and \eqref{Proof; Remainder1} yield
$$
\begin{aligned}
&\Big| \sum_{k=1}^n \iiint_{|y_1|<1}R^f_{N+2}(x_1, y_1, x_2) \,
\tau_{n^{-1/2}}\nu^{k-1}(dx_1)\nn\\
&\hspace{1cm}
(\tau_{n^{-1/2}}\mu^{(k)} - \tau_{n^{-1/2}}\nu)(dy_1)
\tau_{n^{-1/2}}(\mu^{(k+1)}\cdots\mu^{(n)})(dx_2)\Big|\\
&\le \frac{1}{n}C(G) \|D^{(N+2)r}f\|_\infty 
\sum_{k=1}^n \Big(\frac{1}{n}\big( L_{N+2}(\mu^{(k)})+L_{N+2}(\nu) \big)\Big)\nn\\
&\hspace{1cm}
\times(1+|x|^{(N+2)(r-1)}+m_{\nu^{k-1}}^{(N+2)(r-1)})\\
&\le \frac{1}{n^2}C(G) \|D^{(N+2)r}f\|_\infty 
\sum_{k=1}^n n^{-N/2}(m^{N+2}_{\mu^{(k)}}+m^{N+2}_\nu)\nn\\
&\hspace{1cm}
\times(1+|x|^{(N+2)(r-1)}+m_{\nu^{k-1}}^{(N+2)(r-1)})=O(n^{-N/2}),
\end{aligned}
$$
where we have shown that the contribution of the remainder term in \eqref{Proof; Taylor1} is up to $O(n^{-N/2})$. 

\vspace{2mm}
\noindent
{\bf Step 3.} We split the summation on the right-hand side of \eqref{Proof; Taylor1} into 
$$
\sum_{d(I_1) \le N+1} \widehat{S}^{I_1}_{(2)}f(x \bullet x_1 \bullet \bm{1}_G \bullet x_2)
y_1^{I_1} =\Big(\sum_{d(I_1) \le 2}+\sum_{3 \le d(I_1) \le N+1}\Big)
\widehat{S}^{I_1}_{(2)}f(x \bullet x_1 \bullet \bm{1}_G \bullet x_2)y_1^{I_1}.
$$
Let us consider the case $d(I_1) \le 2$. 
Since it holds that $m_{\mu^{(k)}}^{I_1}=m_\nu^{I_1}$
for $k=1, 2, \dots, n$ when $d(I_1)=1$, we have
$$
\begin{aligned}
&\Big|\int_{|y_1|<1}y_1^{I_1} \, 
(\tau_{n^{-1/2}}\mu^{(k)}-\tau_{n^{-1/2}}\nu\big)(dy_1)\Big|\\
&=
\Big|\int_{|y_1| \ge 1}y_1^{I_1} \, 
(\tau_{n^{-1/2}}\mu^{(k)}-\tau_{n^{-1/2}}\nu\big)(dy_1)\Big|
\le \frac{1}{n}\big(\Lambda_2(\mu^{(k)})+\Lambda_2(\nu)\big)
\end{aligned}
$$
for $k=1, 2, \dots, n$. 
Therefore, Lemma \ref{estimate of f} and \eqref{truncation} imply that 
$$
\begin{aligned}
&\Big| \sum_{k=1}^n \sum_{d(I_1) \le 2}\iiint_{|y_1|<1} 
\widehat{S}^{I_1}_{(2)}f(x \bullet x_1 \bullet \bm{1}_G \bullet x_2)y_1^{I_1}  \\
&\hspace{1cm}\times \tau_{n^{-1/2}}\nu^{k-1}(dx_1)
(\tau_{n^{-1/2}}\mu^{(k)} - \tau_{n^{-1/2}}\nu)(dy_1)
\tau_{n^{-1/2}}(\mu^{(k+1)}\cdots\mu^{(n)})(dx_2)\Big|\\
&\le \frac{1}{n}C(G) \|D^{2r}f\|_\infty 
\sum_{k=1}^n \sum_{d(I_1) \le 2}\Big(\frac{1}{n}
\big( \Lambda_{2}(\mu^{(k)})+\Lambda_{2}(\nu) \big)\Big)\nn\\
&\hspace{1cm}
\times(1+|x|^{2(r-1)}+m_{\nu^{k-1}}^{2(r-1)})\\
&\le \frac{1}{n^2}C(G) \|D^{2r}f\|_\infty 
\sum_{k=1}^n \sum_{d(I_1) \le 2}n^{-N/2}(m^{N+2}_{\mu^{(k)}}+m^{N+2}_\nu)\nn\\
&\hspace{1cm}
\times(1+|x|^{2(r-1)}+m_{\nu^{k-1}}^{2(r-1)})=O(n^{-N/2}).
\end{aligned}
$$

\noindent
{\bf Step 4.} We consider the case where $3 \le d(I_1) \le N+1$. 
In the same way as \eqref{Proof; measure-decomp}, we have
\begin{align}
&\sum_{k=1}^n \tau_{n^{-1/2}}\nu^{k-1}
(\tau_{n^{-1/2}}\mu^{(k)} - \tau_{n^{-1/2}}\nu)
\tau_{n^{-1/2}}(\mu^{(k+1)}\cdots\mu^{(n)}) \nn\\
&=\sum_{k=1}^n \tau_{n^{-1/2}}\nu^{k-1}(\tau_{n^{-1/2}}\mu^{(k)} - \tau_{n^{-1/2}}\nu)
\tau_{n^{-1/2}}\nu^{n-k}\nn\\
&\hspace{1cm}
+\sum_{k=1}^n \sum_{\ell=1}^{n-k}\tau_{n^{-1/2}}\nu^{k-1}(\tau_{n^{-1/2}}\mu^{(k)} - \tau_{n^{-1/2}}\nu) \nn\\
&\hspace{1cm}\times \tau_{n^{-1/2}}\nu^{\ell-1}
(\tau_{n^{-1/2}}\mu^{(k+\ell)}-\tau_{n^{-1/2}}\nu)\tau_{n^{-1/2}}
(\mu^{(k+\ell+1)}\cdots\mu^{(n)})\nn\\
&=:\mathfrak{M}_{1}^{(n)}+\mathfrak{M}_{2}^{(n)}. 
\label{Proof; Measure-Decomp2}
\end{align}
Then we obtain
$$
\begin{aligned}
&\Big|  \sum_{3 \le d(I_1) \le N+1}
\iiint\hspace{-2.5mm}\iint_{|y_1|<1, |y_2| \ge 1}
\widehat{S}^{I_1}_{(2)}f(x \bullet x_1 \bullet \bm{1}_G \bullet x_2 \bullet y_2 \bullet x_3)\nn\\
&\hspace{1cm}\times y_1^{I_1}\,
\mathfrak{M}_{2}^{(n)}(dx_1dy_1dx_2dy_2dx_3)\Big| \\
&\le \frac{1}{n^{3/2}}C(G)\sum_{k=1}^n \sum_{\ell=1}^{n-k}\sum_{3 \le d(I_1) \le N+1}\|D^{d(I_1)N}f\|_\infty
\Big(\frac{1}{n}\big( L_{d(I_1)}(\mu^{(k)})+L_{d(I_1)}(\nu)\big)\Big)\nn\\
&\hspace{1cm}\times
\Big(\frac{1}{n}\big( \Lambda_0(\mu^{(k+\ell)})+\Lambda_0(\nu)\big)\Big)\\
&\le \frac{1}{n^{7/2}}C(G)\sum_{k=1}^n \sum_{\ell=1}^{n-k}\sum_{3 \le d(I_1) \le N+1}\|D^{d(I_1)N}f\|_\infty n^{-(d(I_1)-2)/2}n^{-N/2}\nn\\
&\hspace{1cm}\times (m^{d(I_1)}_{\mu^{(k)}}+m^{d(I_1)}_{\nu})
(m_{\mu^{k+\ell}}^{N+2}+m_{\nu}^{N+2})\\
&=O(n^{-N/2})
\end{aligned} 
$$
by applying Lemma \ref{estimate of f} and \eqref{truncation}. 

\vspace{2mm}
\noindent
{\bf Step 5.} 
On the set $\{|y_1| <1, \, |y_2|<1\}$, 
we apply the right stratified Tayler formula 
up to the order $N+3-d(I_1)$ to the function 
$\widehat{S}^{I_1}_{(2)}f(x \bullet x_1 \bullet \bm{1}_G \bullet x_2 \bullet y_2 \bullet x_3)$
at $\bm{1}_G$ with respect to the variable $y_2$. 
Then, for fixed $x_1, x_2, x_3 \in G$, we have 
\begin{align}\label{Proof; Taylor2}
&\widehat{S}^{I_1}_{(2)}f(x \bullet x_1 \bullet \bm{1}_G \bullet x_2 \bullet y_2 \bullet x_3)\nn\\
&=\sum_{d(I_2) \le N+3-d(I_1)}
\widehat{S}^{I_2}_{(4)}\widehat{S}^{I_1}_{(2)}
f(x \bullet x_1 \bullet \bm{1}_G \bullet x_2 \bullet \bm{1}_G \bullet x_3)y_2^{I_2}
+R^{\widehat{S}^{I_1}_{(2)}f}_{N+3-d(I_1)},
\end{align}
where the remainder term $R^{\widehat{S}^{I_1}_{(2)}f}_{N+3-d(I_1)}$ satisfies
$$
\begin{aligned}
&|R^{\widehat{S}^{I_1}_{(2)}f}_{N+3-d(I_1)}(x_1, x_2, y_2, x_3)| \nn\\
&\le C(G)|y_2|^{N+4-d(I_1)} 
\sup\Big\{ |\widehat{\mathfrak{a}}^{I_2}_{(4)}
\widehat{S}^{I_1}_{(2)}f(x \bullet x_1 \bullet  \bm{1}_G \bullet x_2 \bullet y_2' \bullet x_3)| \\
&\hspace{1cm}\Big| d(I_2)=N+4-d(I_1), \, |y_2'| \le b(G)^{N+4-d(I_1)}|y_2|\Big\}.  
\end{aligned}
$$
Since it follows from Lemma \ref{estimate of f} that 
$$
\begin{aligned}
&|R^{\widehat{S}^{I_1}_{(2)} f}_{N+3-d(I_1)}(x_1, x_2, y_2, x_3)|\\
&\le C(G)|y_2|^{N+4-d(I_1)}
(1+|x \bullet x_1\bullet x_2 \bullet y_2|^{(N+4-d(I_1))(r-1)})\|D^{(N+4-d(I_1))r}f\|_\infty,
\end{aligned}
$$
we have
$$
\begin{aligned}
&\Big| \sum_{3 \le d(I_1) \le N+1} \iiint\hspace{-2.5mm}\iint_{|y_1|<1, |y_2| < 1}
R^{\widehat{S}^{I_1}_{(2)}f}_{N+3-d(I_1)}(x_1, x_2, y_2, x_3)\nn\\
&\hspace{1cm}\times y_1^{I_1}
 \, \mathfrak{M}_{2}^{(n)}(dx_1dy_1dx_2dy_2dx_3)\Big| \\
&\le \frac{1}{n^{3/2}}C(G)\sum_{3 \le d(I_1) \le N+1} 
\sum_{k=1}^n \sum_{\ell=1}^{n-k}
\|D^{(N+4-d(I_1))r}f\|_\infty\\
&\hspace{1cm}\times \Big(\frac{1}{n}
\big( L_{d(I_1)}(\mu^{(k)})+L_{d(I_1)}(\nu)\big)\Big)
\Big(\frac{1}{n}\big( L_{N+4-d(I_1)}(\mu^{(k+\ell)})+L_{N+4-d(I_1)}(\nu)\big)\Big) \\
&\hspace{1cm}\times 
\Big\{ 1+|x|^{(N+4-d(I_1))(r-1)}+m^{(N+4-d(I_1))(r-1)}_{\nu^{k-1}} + m^{(N+4-d(I_1))(r-1)}_{\nu^{\ell-1}} \Big\}\\
&\le \frac{1}{n^{7/2}}C(G)\sum_{3 \le d(I_1) \le N+1} \sum_{k=1}^n \sum_{\ell=1}^{n-k}
\|D^{(N+4-d(I_1))r}f\|_\infty n^{-(N+2-d(I_1))/2}n^{-(d(I_1)-2)/2}\\
&\hspace{1cm}\times (m_{\mu^{(k+\ell)}}^{N+4-d(I_1)}+m_{\nu}^{N+4-d(I_1)})
\Big\{ 1+|x|^{(N+4-d(I_1))(r-1)}\nn\\
&\hspace{1cm}+m^{(N+4-d(I_1))(r-1)}_{\nu^{k-1}} + m^{(N+4-d(I_1))(r-1)}_{\nu^{\ell-1}} \Big\}
=O(n^{-N/2})
\end{aligned}
$$
by using Lemma \ref{estimate of f} and \eqref{truncation}.

\vspace{2mm}
\noindent
{\bf Step 6.} 
We split the summation $\sum_{d(I_2) \le N+3-d(I_1)}$ as in \eqref{Proof; Taylor2} 
into the summations $\sum_{d(I_2) \le 2}$ and $\sum_{3 \le d(I_2) \le N+3-d(I_1)}$. 
Then, in the same way as {\bf Step 3}, we obtain
$$
\begin{aligned}
&\Big| \sum_{3 \le d(I_1) \le N+1} \sum_{d(I_2) \le 2}
\iiint\hspace{-2.5mm}\iint_{|y_1|<1, |y_2| < 1}
 \widehat{S}_{(4)}^{I_2}\widehat{S}_{(2)}^{I_1}
f(x \bullet x_1 \bullet \bm{1}_G \bullet x_2 \bullet \bm{1}_G \bullet x_3)\\
&\hspace{1cm}\times 
y_1^{I_1}y_2^{I_2}
\, \mathfrak{M}_2^{(n)}(dx_1dy_2dx_1dy_2dx_3)\Big|=O(n^{-1/2}).
\end{aligned}
$$
As for the terms corresponding to the summations over 
$3 \le d(I) \le N+1$ and
$3 \le d(I_2) \le N+3-d(I_1)$, 
that is, over $d(I_1)+d(I_2) \le N+3$ with $d(I_1), d(I_2) \ge 3$, 
their estimates can be done by applying the similar idea to the one 
at {\bf Step 4}. 
Namely, we replace the measure $\tau_{n^{-1/2}}(\mu^{(k+\ell+1)} \cdots \mu^{(n)})$
in \eqref{Proof; Measure-Decomp2} by
$$
\begin{aligned}
&\tau_{n^{-1/2}}(\mu^{(k+\ell+1)} \cdots \mu^{(n)})\\
&=\tau_{n^{-1/2}}\nu^{n-k-\ell} + \sum_{j=1}^{n-k-\ell}
\tau_{n^{-1/2}}\nu^{j-1}\tau_{n^{-1/2}}(\mu^{(k+\ell+j)} - \nu) 
\tau_{n^{-1/2}}(\mu^{(k+\ell+j+1)} \cdots \mu^{(n)}).
\end{aligned}
$$
We consider the integration with respect to the 
measure given by the second term in 
the decomposition above. 
Then the same argument as the one at {\bf Step 5} can be done here. 
It is clearly observed that this procedure finishes at $N$ times and 
then the integrand is given by 
$$
\widehat{S}^{I_{N-1}}_{(2(N-1))}\widehat{S}^{I_{N-2}}_{(2(N-2))} 
\cdots \widehat{S}^{I_1}_{(2)}
f(x \bullet x_1 \bullet \bm{1}_G \bullet x_2 \bullet \bm{1}_G \bullet \cdots \bullet 
\bm{1}_G \bullet x_{N})y_1^{I_1}y_2^{I_2} \cdots y_{N-1}^{I_{N-1}}
$$
by repeating the use of the right stratified Taylor formula, 
where $d(I)=d(I_2)=\cdots=d(I_{N-1})=3$ and the domain of the integral 
is $\{|y_1|<1, |y_2|<1, \cdots, |y_{N-1}|<1\}$. 
In this procedure, the estimates of integrals with respect to the measures
$$
\begin{aligned}
&\sum_{i_1+i_2+\cdots+i_\ell=n-\ell+1}\tau_{n^{-1/2}}\nu^{i_1}
\tau_{n^{-1/2}}(\mu^{(i_1+1)}-\nu) \tau_{n^{-1/2}}\nu^{i_2}
\tau_{n^{-1/2}}(\mu^{(i_1+i_2+2)}-\nu) \\
&\hspace{1cm} \times \cdots \times \tau_{n^{-1/2}}\nu^{i_{\ell-1}}
\tau_{n^{-1/2}}(\mu^{(n-i_\ell)}-\nu)\tau_{n^{-1/2}}\nu^{i_\ell}, \qquad 
\ell=2, 3, \dots, N
\end{aligned}
$$
including e.g., $\mathfrak{M}_{1}^{(n)}$ have not been done, though 
the rest term is easily shown to be up to $O(n^{-1/2})$ by applying
the right stratified Taylor formula up to order 2 
and the same way as that of {\bf Step 2}. 

\vspace{2mm}
\noindent
{\bf Step 7}. At this final step, we consider the terms written by 
\begin{align}
&\sum_{\substack{i_1+i_2+\cdots+i_\ell \\ =n-\ell+1}}
\sum_{\substack{d(I_1)+\cdots+d(I_{\ell-1}) \le N+2\ell-3 \\ d(I_1), \dots, d(I_{\ell-1}) \ge 3}}
\int\hspace{-1mm}\cdots\hspace{-1mm}\int_{G^\ell} \widehat{S}^{I_{\ell-1}}_{(2(\ell-1))}\widehat{S}^{I_{\ell-2}}_{(2(\ell-2))} 
\cdots \widehat{S}^{I_1}_{(2)}
f(x \bullet x_1 \nn\\
&\hspace{0.5cm} \bullet \bm{1}_G \bullet x_2 \bullet 
\bm{1}_G \bullet \cdots \bullet \bm{1}_G \bullet x_\ell)
\tau_{n^{-1/2}}\nu^{i_1}(dx_1)\tau_{n^{-1/2}}\nu^{i_2}(dx_2) 
\cdots \tau_{n^{-1/2}}\nu^{i_\ell}(dx_\ell)\nn\\
&\hspace{0.5cm}\times \Big(\int_{|y_1|<1}y_1^{I_1}\tau_{n^{-1/2}}(\mu^{(i_1+1)}-\nu)(dy_1)\Big) \cdots \Big(\int_{|y_{\ell-1}|<1}y_{\ell-1}^{I_{\ell-1}}\tau_{n^{-1/2}}(\mu^{(n-i_\ell)}-\nu)(dy_{\ell-1})\Big)
\label{Proof; final terms}
\end{align}
for $\ell=2, 3, \dots, N$. 
Since it holds that 
$$
\begin{aligned}
&\int_{|y_k|<1}y_k^{I_k}\tau_{n^{-1/2}}(\mu^{(i_1+\cdots+i_k+k)}-\nu)(dy_k)\\
&=\int_{G}y_k^{I_k}\tau_{n^{-1/2}}(\mu^{(i_1+\cdots+i_k+k)}-\nu)(dy_k)
-\int_{|y_k| \ge 1}y_k^{I_k}\tau_{n^{-1/2}}(\mu^{(i_1+\cdots+i_k+k)}-\nu)(dy_k)\\
&=n^{-d(I_k)/2}(m^{I_k}_{\Phi_0} - m^{I_k}_{\nu})
-\int_{|y_k| \ge 1}y_k^{I_k}\tau_{n^{-1/2}}(\mu^{(i_1+\cdots+i_k+k)}-\nu)(dy_k)
\end{aligned}
$$
and 
$$
\begin{aligned}
&\prod_{k=1}^{\ell-1}\Big|\int_{|y_k| \ge 1}y_k^{I_k}\tau_{n^{-1/2}}(\mu^{(i_1+\cdots+i_k+k)}-\nu)(dy_k) \Big|\\
&\le \prod_{k=1}^{\ell-1} n^{-d(I_k)/2} \cdot \frac{1}{n}\Big(\Lambda_{d(I_k)}(\mu^{(i_1+\cdots+i_k+k)})+\Lambda_{d(I_k)}(\nu)\Big)\\
&\le n^{-3(\ell-1)/2}n^{-(\ell-1)}\prod_{k=1}^{\ell-1} 
n^{-N/2}(m^{N+2}_{\mu^{(i_1+\cdots+i_k+k)}}+m^{N+2}_\nu)\\
&=n^{-(N+5)(\ell-1)/2}\prod_{k=1}^{\ell-1} 
(m^{N+2}_{\mu^{(i_1+\cdots+i_k+k)}}+m^{N+2}_\nu)=O(n^{-N/2})
\end{aligned}
$$
for $k=1, 2, \dots, \ell$ and $\ell=2, 3, \dots, N$, we may replace each integral 
$$
\int_{|y_k|<1}y_k^{I_k}\tau_{n^{-1/2}}(\mu^{(i_1+\cdots+i_k+k)}-\nu)(dy_k), 
\qquad k=1, 2, \dots, \ell, \,\, \ell=2, 3, \dots, N
$$ 
in \eqref{Proof; final terms} by $n^{-d(I_k)/2}(m^{I_k}_{\Phi_0} - m^{I_k}_{\nu})$. 
Therefore, we obtain
\begin{align}
&\sum_{\ell=2}^N\sum_{\substack{i_1+i_2+\cdots+i_\ell \\=n-\ell+1}}
\sum_{\substack{d(I_1)+\cdots+d(I_{\ell-1}) 
\le N+2\ell-3 \\ d(I_1), \dots, d(I_{\ell-1}) \ge 3}}
\int\hspace{-1mm}\cdots\hspace{-1mm}\int_{G^\ell}
\widehat{S}^{I_{\ell-1}}_{(2(\ell-1))}\widehat{S}^{I_{\ell-2}}_{(2(\ell-2))} 
\cdots \widehat{S}^{I_1}_{(2)}
f(x \bullet x_1\nn\\
&\hspace{1cm}
 \bullet \bm{1}_G \bullet  \cdots \bullet \bm{1}_G \bullet x_\ell)
\nu_{i_1/n}^{\Phi_0}(dx_1)\nu_{i_2/n}^{\Phi_0}(dx_2) 
\cdots \nu_{i_\ell/n}^{\Phi_0}(dx_\ell) \nn\\
&\hspace{1cm}\times 
\prod_{k=1}^{\ell-1}n^{-d(I_k)/2}(m^{I_k}_{\Phi_0} - m^{I_k}_{\nu})\nn\\
&=\frac{\xi_1}{n^{1/2}}+\frac{\xi_2}{n^{2/2}}+\cdots+\frac{\xi_{N-1}}{n^{(N-1)/2}},
\label{pre-expansion}
\end{align}
where we used $\tau_{n^{-1/2}}\nu^i=\nu_{i/n}^{\Phi_0}$ for $i=1, 2, \dots$ and 
the each coefficient $\xi_j$, $j=1, 2, \dots, N-1$, depends on 
$x$, $G$,  $f$, $p$ and $\Phi_0$. 
The explicit representation of each $\xi_j$, $j=1, 2, \dots, N-1$, is 
discussed in the next subsection. By putting it all together, 
we obtain the desired expansion \eqref{Edgeworth} 
and this completes the proof of Theorem \ref{main}. 
\end{proof}

\subsection{Explicit representations of coefficients in the Edgeworth expansion}

We give explicit representations of the coefficients $\xi_j=\xi_j(x, G, f, p, \Phi_0),$ 
$j=1, 2, \dots, N-1$, appearing in the Edgeworth expansion in Theorem \ref{main}. 
To obtain them, we need to 
pick up the terms of order $n^{-j/2}$
on the left-hand side of \eqref{pre-expansion}. 
For this sake, we need to use the following  multidimensional 
version of the so-called Euler--Maclaurin summation formula. 
We refer to \cite{BP} for more details. 

\begin{lm}[cf.\,{\cite[Theorem 1]{Pap2}}]
\label{EM}
Let $\ell \ge 2$ and $\Delta(\ell)$ be a simplex defined by
$$
\Delta(\ell):=\{t=(t_1, t_2, \dots, t_\ell) \in \mathbb{R}^{\ell} \, | \, 
t_1+t_2+\cdots+t_\ell=1, \, t_1, t_2, \dots, t_\ell \ge 0\}. 
$$
Suppose that a function $F : \Delta(\ell) \to \mathbb{R}$ has
continuous partial derivatives of all orders $J\in \mathcal{I}$
with $|J| \le (\ell-1)s$ for some $s \in \mathbb{N}$.  Then we have
$$
\begin{aligned}
&\frac{1}{n^{\ell-1}}\sum_{\substack{i_1+i_2+\cdots+i_\ell \\=n-\ell+1}}
F\Big(\frac{i_1}{n}, \frac{i_2}{n}, \dots, \frac{i_\ell}{n}\Big)\\
&=\sum_{i=2}^\ell \sum_{k=0}^{s-1} \frac{1}{k! n^k}
\int_{\Delta(i)}B_k^{(\ell, i)}(\del_1, \del_2, \dots, \del_\ell)
F(t, \underbrace{0, 0, \dots, 0}_{(\ell-i)\text{-times}}) \, dt+R_n^{(\ell, s)},
\end{aligned}
$$
where each $B_j^{(\ell, i)}$ is a polynomial defined by the following generating function
$$
\sum_{j=0}^\infty \frac{B_j^{(\ell, i)}(t_1, t_2, \dots, t_\ell)}{j!}x^{j}
=(-1)^{\ell-i}x^{\ell-1}\sum_{q=i}^{\ell}\frac{(t_1-t_q)(t_2-t_q)\cdots(t_i-t_q)}{\dis\prod_{\substack{\alpha=1, 2, \dots, k, \\ \alpha \neq q}}(e^{xt_\alpha}-e^{xt_q})}
$$
and 
$R_n^{(\ell, s)}$ satisfies 
$|R_n^{(\ell, s)}| \le C(\ell, s, F)n^{-s}$ for some postive constant $C(\ell, s, F)>0$. 
\end{lm}

We now apply Lemma \ref{EM} to \eqref{pre-expansion} in the case where 
$$
\begin{aligned}
&F\Big(\frac{i_1}{n}, \frac{i_2}{n}, \dots, \frac{i_\ell}{n}\Big)\\
&=\sum_{\substack{d(I_1)+\cdots+d(I_{\ell-1}) 
\le N+2\ell-3 \\ d(I_1), \dots, d(I_{\ell-1}) \ge 3}}
\int\hspace{-1mm}\cdots\hspace{-1mm}\int_{G^\ell}
\widehat{S}^{I_{\ell-1}}_{(2(\ell-1))}\widehat{S}^{I_{\ell-2}}_{(2(\ell-2))} 
\cdots \widehat{S}^{I_1}_{(2)}
f(x \bullet x_1 \bullet \bm{1}_G \bullet  \cdots \bullet \bm{1}_G \bullet x_\ell)\nn\\
&\hspace{0.5cm}\times 
\nu_{i_1/n}^{\Phi_0}(dx_1)\nu_{i_2/n}^{\Phi_0}(dx_2) 
\cdots \nu_{i_\ell/n}^{\Phi_0}(dx_\ell).
\end{aligned}
$$
Then, we have the following explicit representations of coefficients $\xi_1, \xi_2, \dots, \xi_{N-1}$
in terms of the centered Gaussian semigroup $(\nu_t^{\Phi_0})_{t \ge 0}$ associated with 
the infinitesimal generator $\A(\Phi_0)$. 

\begin{pr}\label{Edgeworth-coef}
Let $\{\xi_j=\xi_j(x, G, f, p, \Phi_0)\}_{j=1}^{N-1}$ be as in Theorem {\rm\ref{main}}. 
Then, for every $j=1, 2, \dots, N-1$, we obtain
\begin{align}
\xi_j&=\sum_{i=1}^j \sum_{\ell=2}^{i+1} \sum_{q=i-\ell+1}^{[(j-i)/2]}
\sum_{\substack{d(I_1)+\cdots+d(I_i)=j-2q+2\ell-2 \\ d(I_1), \dots, d(I_i) \ge 3}} \nn\\ &\hspace{1cm}\Bigg\{ \int_{\Delta(i)}\int_{G^{i}}
\mathcal{D}_{i, \ell, q}^{I_1, I_2, \dots, I_i}f(x \bullet x_1 \bullet \bm{1}_G \bullet  \cdots \bullet \bm{1}_G \bullet x_\ell \bullet \bm{1}_G \bullet 
\bm{1}_G \bullet \cdots \bullet \bm{1}_G \bullet \bm{1}_G)\nn\\
&\hspace{1cm}\times \nu_{t_1}^{\Phi_0}(dx_1)\nu_{t_2}^{\Phi_0}(dx_2) \cdots 
\nu_{t_\ell}^{\Phi_0}(dx_\ell) \, dt_1dt_2 \cdots dt_\ell \times \prod_{k=1}^i 
(m^{I_k}_{\Phi_0} - m^{I_k}_{\nu})\Bigg\},
\label{coef}
\end{align}
where each $\mathcal{D}_{i, \ell, q}^{I_1, I_2, \dots, I_i}$ is a 
differential operator defined by
$$
\mathcal{D}_{i, \ell, q}^{I_1, I_2, \dots, I_i}
=\frac{1}{q!}B_q^{(i+1, \ell)}(\A_{(1)}, \A_{(3)}, \dots, \A_{(2i+1)})
\widehat{S}^{I_{i}}_{(2i)}\widehat{S}^{I_{i-1}}_{(2(i-1))} 
\cdots \widehat{S}^{I_1}_{(2)}. 
$$
\end{pr}

\section{{\bf Berry--Esseen type bound via Trotter's approximation theorem}}

By applying Theorem \ref{main} in the case where $N=2$, we immediately establish the 
so-called Berry--Esseen type bound for 
the scaled random walk $\{\tau_{n^{-1/2}}(\xi_n)\}_{n=1}^\infty$, 
which gives the rate of convergence of the discrete semigroup generated by  
$\{\tau_{n^{-1/2}}(\xi_n)\}_{n=1}^\infty$ to the heat semigroup 
$(\e^{t\A(\Phi_0)})_{t \ge 0}$ whose infinitesimal 
generator is $\A(\Phi_0)$. 

\begin{co}[Berry--Esseen type bound]
\label{BE-bound}
For every $f \in C^{9r}(G)$, we have 
$$
\|\mathcal{L}^{[nt]} \mathcal{P}^H_{n^{-1/2}}f
- \mathcal{P}^H_{n^{-1/2}}\e^{t\mathcal{A}(\Phi_0)}f\|_\infty 
\le \frac{C}{n^{1/2}}, \qquad n \in \mathbb{N},
$$
for some positive constant $C=C(G, f, p, \Phi_0, t)>0$. 
\end{co}


In this section, we give an alternative proof of Corollary \ref{BE-bound}
from functional analytic point of view. 
Recall that the proof of Proposition \ref{CLT} given in \cite{IKN}
heavily depends on the celebrated Trotter's semigroup approximation theorem.  
It provides a sufficient condition for the convergence of a sequence 
of operator semigroups in terms of the corresponding sequence of infinitesimal generators.
See \cite{Trotter} and \cite{Kurtz} for more details. 

Let $(B_n, \|\cdot\|_{B_n}), \, n \in \mathbb{N}$, and $(E, \|\cdot\|_E)$ be 
 Banach spaces. Let $P_n : E \to B_n, \, n \in \mathbb{N},$ be a bounded linear operator
 satisfying $\|f_n-P_nf\|_{B_n} \to 0$ as $n \to \infty$. 
 We say that the sequence of pairs $\{(B_n, P_n)\}_{n=1}^\infty$ approximates $E$ 
if  $\|P_n f\|_{B_n} \to \|f\|_E$ as $n \to \infty$ for every $f \in E$.
 Then, Trotter's approximation theorem can be stated as follows:

\begin{pr}[cf.\,{\cite{Trotter, Kurtz}}]
\label{Trotter's approximation}
Let $T_n, \, n \in \mathbb{N},$ be a bounded linear operator on $B_n$ with $\|T_n\| \le 1$.
Let $\{\ell(n)\}_{n=1}^\infty$ be a sequence of positive numbers and 
$\mathfrak{A}_n:=(T_n-I)/\ell(n)$ for $n \in \mathbb{N}.$ 
Suppose that $\ell(n) \to 0$ as $n \to \infty$ and $\mathfrak{A}$ is defined by the closure of 
the limit $\lim_{n \to \infty}\mathfrak{A}_n$. 
If the domain $\mathrm{Dom}(\mathfrak{A})$ is dense in $E$ and 
the range $\mathrm{Ran}(\lambda_0-\mathfrak{A})$ is dense in $E$ for some $\lambda_0>0$, 
then there exists a $C_0$-semigroup $(\mathcal{T}_t)_{t \ge 0}$ on $E$ such that
$$
\lim_{n \to \infty}\|T_n^{[t/\ell(n)]}P_n f - P_n \mathcal{T}_t f\|_{B_n}=0, \qquad t \ge 0. 
$$
\end{pr}

Afterwards, Campiti and Tacelli  \cite[Theorem 1.1]{CT}
determined a rate of convergence of the 
Trotter's theorem 
in the case where $B_n \equiv E$ for $n \in \mathbb{N}$. 
The following theorem is a certain refinement of Proposition \ref{Trotter's approximation}
and is also regarded as an extension of \cite[Theorem 1.1]{CT}.

\begin{tm}[{\cite[Theorem 1.3]{Namba}}]
\label{Trotter-refine}
Let $T_n, \, n \in \mathbb{N},$ be a bounded linear operator on $B_n$ satisfying 
\begin{equation}\label{cond1}
\|T_n^k\| \le Me^{\omega k/n}, \qquad n, k \in \mathbb{N},
\end{equation}
for some $M \ge 1$ and $\omega \ge 0$. 
Suppose that $\mathfrak{D}$ is a dense subspace of $E$
and $\mathfrak{A} : (\mathfrak{D} \subset)\mathrm{Dom}(\mathfrak{A}) \to E$
is a linear operator. 
If $\mathrm{Ran}(\lambda-\mathfrak{A})$ is dense in $E$ for some $\lambda>\omega$,
then the closure of $(\mathfrak{A}, \mathfrak{D})$
 generates a $C_0$-semigroup
$(\mathcal{T}_t)_{t \geq 0}$ on $E$ satisfying $\|\mathcal{T}_t\| \le Me^{\omega t}$
for $t \ge 0$. 
Moreover, suppose that 
\begin{equation}\label{CT1}
\|n(T_n-I)P_nf\|_{B_n} \leq \varphi_n(f), \qquad f \in \mathfrak{D},
\end{equation} 
and the following estimate of the Voronovskaja-type formula holds:
\begin{equation}\label{CT2}
\|n(T_n -I)P_n f - P_n\mathfrak{A} f\|_{B_n} \leq \psi_n(f), \qquad f \in \mathfrak{D},
\end{equation} 
where  $\varphi_n, \psi_n : \mathfrak{D} \to [0, \infty)$
are semi-norms on the subspace $\mathfrak{D}$ 
with $\lim_{n \to \infty}\psi_n(f)=0$ for $f \in \mathfrak{D}$. 
Then, for every $t \geq 0$ and for every increasing 
$\{k(n)\}_{n=1}^\infty$ of positive integers, we have
\begin{align}\label{rate of convergence}
&\|T_n^{k(n)}P_nf - P_n\mathcal{T}_t f\|_{B_n} \nn\\
&\le M\exp(2\omega e^{\omega/n}k(n)/n)
\Big(\frac{\omega}{n}\frac{k(n)}{n} +\frac{\sqrt{k(n)}}{n} \Big)\varphi_n(f)\nn\\
&\hspace{1cm}
+M\exp(\omega t_ne^{\omega/n})\Big|\frac{k(n)}{n}-t\Big|\varphi_n(f)\nn\\
&\hspace{1cm}+M\exp(\omega te^{\omega/n})\int_0^t 
\exp(-\omega se^{\omega/n})\psi_n(\mathcal{T}_s f) \, ds
\end{align}
for all 
$f \in \mathfrak{D}_0:=\{g \in \mathfrak{D} \, 
| \, \mathcal{T}_tg \in \mathfrak{D}, \, t \ge 0\}$, 
where we put $t_n:=\max\{t, k(n)/n\}$. 
\end{tm}

See \cite{Namba} for some typical applications as well as its complete proof.  
We here give a 
proof of the following Proposition \ref{BE-refine} 
by a simple application of Theorem 
\ref{Trotter-refine}, though the function space should be supposed to be 
$$
\mathfrak{D}:=C_\infty^\infty(G)= \bigcap_{k=1}^\infty \Bigg\{ f \in C_\infty(G) \, : \, 
\lim_{|x| \to \infty} \mathfrak{a}^I f(x)=0, \, 
I \in \mathcal{I}, \, d(I)=k\Bigg\},
$$
In this sense, Proposition \ref{BE-refine}  
is weaker than Corollary \ref{BE-bound}. 
Nevertheless, it is worth mentioning here since it is proved by using
not any probabilistic techniques but functional analytic ones. 

\begin{pr}
For every $f \in \mathfrak{D}=C_\infty^\infty(G)$, we have 
\begin{equation}\label{BE-refine}
\|\mathcal{L}^{[nt]} \mathcal{P}^H_{n^{-1/2}}f
- \mathcal{P}^H_{n^{-1/2}}\e^{t\mathcal{A}(\Phi_0)}f\|_\infty 
\le \frac{C}{n^{1/2}}, \qquad n \in \mathbb{N},
\end{equation}
for some positive constant $C=C(G, f, p, \Phi_0, t)>0$. 
\end{pr}

\begin{proof}[Proof]
Let us take $(B_n, \|\cdot\|_{B_n})=(C_\infty(X), \|\cdot\|_\infty)$ for $n \in \mathbb{N}$ 
and $(E, \|\cdot\|_{E})=(C_\infty(G), \|\cdot\|_\infty)$. Then 
$\{(C_\infty(X), \mathcal{P}^H_{n^{-1/2}})\}_{n=1}^\infty$ approximates the Banach space $C_\infty(G)$. 
We also take 
$\mathfrak{D}=C_\infty^\infty(G)$, which is a dense subspace of $C_\infty(G)$. 
We define a sequence of bounded linear operators $\{T_n\}_{n=1}^\infty$
on $C_\infty(X)$ by 
$T_n =\mathcal{L} $ for $n \in \mathbb{N}$. 
Note that  $\|\LL^n\| \le 1$ for 
$n \in \mathbb{N}$, that is, $M=1$ and $\omega=0$. 
Moreover, we take $\mathfrak{A}=\A(\Phi_0)$, 
which satisfies that $\mathrm{Ran}(\lambda - \A(\Phi_0))$ is dense in $C_\infty(G)$
for some $\lambda>0$ (cf.\,\cite[page 304]{Rob}). 

We now show that 
\begin{equation}\label{phi}
\|n(\mathcal{L}-I)\mathcal{P}^H_{n^{-1/2}} f \|_\infty \leq \varphi_n(f)=
\|\A(\Phi_0) f\|_\infty + 
\frac{C}{n^{1/2}} \Big(\max_{e \in E_0}|d\Phi_0(\widetilde{e})|^3\Big)\|D^3f\|_{\infty}
\end{equation}
and 
\begin{equation}\label{psi}
\|n(\mathcal{L}-I)\mathcal{P}^H_{n^{-1/2}} f - \mathcal{P}^H_{n^{-1/2}}\A(\Phi_0) f\|_\infty \leq\psi_n(f)=
\frac{C}{n^{1/2}} \Big(\max_{e \in E_0}|d\Phi_0(\widetilde{e})|^3\Big)\|D^3f\|_{\infty}
\end{equation}
for every $f \in \mathfrak{D}$. 
Indeed, we apply the right stratified Taylor formula up to order 2 to 
the function $f$ at $\tau_{n^{-1/2}}(\Phi_0(x))$.
Then, it follows from
$m_{\Phi_0}^I=0$ for $I \in \mathcal{I}$ with $d(I)=1$ (Proposition \ref{phi-moment}) that  
$$
\begin{aligned}
n(\LL -I)\mathcal{P}^H_{n^{-1/2}}f(x) &= n^{1/2}\sum_{d(I)=1}\widehat{S}^I f\Big(\tau_{n^{-1/2}}
\big(\Phi_0(x)\big)\Big) m_{\Phi_0}^I \\
&\hspace{1cm}+
\sum_{d(I)=2}\widehat{S}^I f\Big(\tau_{n^{-1/2}}
\big(\Phi_0(x)\big)\Big) m_{\Phi_0}^I  +n  \mathbb{E}^{x, 1}[R_3^f]\\
&=\sum_{d(I)=2}\widehat{S}^If\Big(\tau_{n^{-1/2}}
\big(\Phi_0(x)\big)\Big)m_{\Phi_0}^I + n\mathbb{E}^{x, 1}[R_3^f], \qquad x \in V,
\end{aligned}
$$
where the remainder term $R_3^f$ satisfies 
$$
\begin{aligned}
\mathbb{E}^{x, 1}[|R^f_{3}|] &\le 
\frac{C}{n^{3/2}}\mathbb{E}^{x, 1}\Bigg[|d\Phi_0(e)|^{3}
\sup\Bigg\{ \Big|\mathfrak{a}^I f\Big(\tau_{n^{-1/2}}(\Phi_0(x)) \bullet z\Big)\Big| \, \nn\\
&\hspace{1cm}: \, d(I)=3, \, |z| \le \frac{b^{3}}{n^{1/2}}|d\Phi_0(e)|\Bigg\}\Bigg]\\
&\le \frac{C}{n^{3/2}} \Big(\max_{e \in E_0}|d\Phi_0(\widetilde{e})|^3\Big)\|D^3f\|_{\infty}.
\end{aligned}
$$
for some positive constants $C>0$. 
We observe that 
$$
\begin{aligned}
&\sum_{d(I)=2}\widehat{S}^I f\Big(\tau_{n^{-1/2}}
\big(\Phi_0(x)\big)\Big) m_{\Phi_0}^I\\
&=\Big(\frac{1}{2}\sum_{i, j=1}^{d_1} \sigma_i(\Phi_0)\sigma_j(\Phi_0)\mathfrak{a}_i\mathfrak{a}_j 
+\sum_{i=d_1+1}^{d_1+d_2} \beta(\Phi_0)|_{\mathfrak{a}_i}\mathfrak{a}_i^2\Big)f\Big(\tau_{n^{-1/2}}
\big(\Phi_0(x)\big)\Big)\\
&=\mathcal{P}^H_{n^{-1/2}}\A(\Phi_0) f(x), \qquad x \in V.
\end{aligned}$$
Hence, we obtain
$$
\begin{aligned}
\|n(\LL -I)\mathcal{P}^H_{n^{-1/2}}f\|_\infty
& \le \|\A(\Phi_0) f\|_\infty + 
\frac{C}{n^{1/2}} \Big(\max_{e \in E_0}|d\Phi_0(\widetilde{e})|^3\Big)\|D^3f\|_{\infty},\\
\|n(\LL -I)\mathcal{P}^H_{n^{-1/2}}f - \mathcal{P}^H_{n^{-1/2}}\A(\Phi_0) f\|_\infty
&\le \frac{C}{n^{1/2}} \Big(\max_{e \in E_0}|d\Phi_0(\widetilde{e})|^3\Big)\|D^3f\|_{\infty}.
\end{aligned}
$$
Therefore, we can apply Theorem \ref{Trotter-refine} in the case where 
$k(n)=[nt]$ for $n \in \mathbb{N}$.  
By using \eqref{phi} and \eqref{psi} 
and by noting $(\e^{t\mathcal{A}(\Phi_0)})(\mathfrak{D}) \subset \mathfrak{D}$ for $t \ge 0$,  
we obtain
$$
\begin{aligned}
&\|\LL^{[nt]}\mathcal{P}^H_{n^{-1/2}}f - \mathcal{P}^H_{n^{-1/2}}\e^{t\A(\Phi_0)}f\|_\infty\\
&\le \frac{\sqrt{[nt]}}{n}\varphi_n(f)+\Big|\frac{[nt]}{n}-t\Big|\varphi_n(f)
+\int_0^t \psi_n(\e^{s\mathcal{A}(\Phi_0)}f) \, ds\\
&\le \frac{t^{1/2}}{n^{1/2}}\varphi_n(f)+\frac{1}{n}\varphi_n(f)
+\int_0^t \psi_n(\e^{s\mathcal{A}(\Phi_0)}f) \, ds
\le \frac{C}{n^{1/2}}, \qquad t \ge 0,
\end{aligned}
$$
for some positive constant $C=C(G, f, p, \Phi_0, t)>0$. 
This means that we have established \eqref{BE-refine} for every $f \in \mathfrak{D}$. 
This completes the proof. 
\end{proof}

Before closing this section, we mention the case where the $\Gamma$-equivariant 
realization $\Phi : X \to G$ is not always modified harmonic. 
Since the proof of Theorem \ref{main} heavily depends on the modified harmonicity of 
the realization, we do not expect to establish the precise Edgeworth expansions 
for the random walks on $X$ without imposing the modified harmonicity.  
However, we now see that the Berry--Esseen type bound 
for arbitrary $\Gamma$-equivariant realization $\Phi : X \to G$
is immediately established. Namely, we obtain the following. 

\begin{tm}[Berry--Esseen type bound for $\Gamma$-equivariant realizations]
\label{BE-bound2}
Let $\Phi : X \to G$ be a $\Gamma$-equivariant realization of $X$ and 
$\mathcal{P}_{n^{-1/2}} : C_\infty(G) \to C_\infty(G), \, n \in \mathbb{N},$ 
be the approximation operator defined by 
$$
\mathcal{P}_{n^{-1/2}}f(x):=f\Big(\tau_{n^{-1/2}}\big(\Phi(x)\big)\Big), \qquad x \in V. 
$$
For every $f \in C_c^\infty(G)$ we have 
$$
\|\mathcal{L}^{[nt]} \mathcal{P}_{n^{-1/2}}f
- \mathcal{P}_{n^{-1/2}}\e^{t\mathcal{A}(\Phi_0)}f\|_\infty 
\le \frac{C}{n^{1/2}}, \qquad n \in \mathbb{N},
$$
for some positive constant $C=C(G, f, p, \Phi, \Phi_0, t)>0$. 
\end{tm}

\begin{proof}
We fix a reference point $x_* \in V$. Then
we may put $\Phi(x_*)=\Phi_{0}(x_*)=\bm{1}_G$ without loss of generality. 
By using the triangular inequality and Corollary \ref{BE-bound}, we have 
$$
\begin{aligned}
&\|\mathcal{L}^{[nt]} \mathcal{P}_{n^{-1/2}}f
- \mathcal{P}_{n^{-1/2}}\e^{t\mathcal{A}(\Phi_0)}f\|_\infty \\
&\le \|\mathcal{L}^{[nt]} \mathcal{P}_{n^{-1/2}}f
- \mathcal{L}^{[nt]} \mathcal{P}^H_{n^{-1/2}}f\|_\infty +
 \|\mathcal{L}^{[nt]} \mathcal{P}^H_{n^{-1/2}}f
- \mathcal{P}^H_{n^{-1/2}}\e^{t\mathcal{A}(\Phi_0)}f\|_\infty\\
&\hspace{1cm}+ \| \mathcal{P}^H_{n^{-1/2}}\e^{t\mathcal{A}(\Phi_0)}f
- \mathcal{P}_{n^{-1/2}}\e^{t\mathcal{A}(\Phi_0)}f\|_\infty\\
&\le  \|  \mathcal{P}_{n^{-1/2}}f
-  \mathcal{P}^H_{n^{-1/2}}f\|_\infty +  \| \mathcal{P}^H_{n^{-1/2}}\e^{t\mathcal{A}(\Phi_0)}f
- \mathcal{P}_{n^{-1/2}}\e^{t\mathcal{A}(\Phi_0)}f\|_\infty + \frac{C}{n^{1/2}}. 
\end{aligned}
$$
Recall that there is an intrinsic left invariant metric on $G$ 
called the {\it Carnot--Carath\'eodory metric} given by 
$$
d_{\mathrm{CC}}(g, h):=\inf \Big\{ \int_0^1 \|\dot{c}(t)\|_{\g^{(1)}} \, dt \, \Big| \, 
c(0)=g, \, c(1)=h, \, \dot{c}(t) \in \g^{(1)}_{c(t)}\Big\}, \qquad g, h \in G,
$$
where $\|\cdot\|_{\g^{(1)}}$ denotes a fixed norm on $\g^{(1)}$ 
and $\g^{(1)}_{c(t)}$ is the evaluation of $\g^{(1)}$
at $c(t)$. 
Due to $f \in C_c^\infty(G)$, we find a positive constant $C>0$ such that 
$$
\begin{aligned}
&|\mathcal{P}_{n^{-1/2}}f(x)-  \mathcal{P}^H_{n^{-1/2}}f(x)|\\
&=\Big|f\Big(\tau_{n^{-1/2}}\big(\Phi(x)\big)\Big) - f\Big(\tau_{n^{-1/2}}\big(\Phi_0(x)\big)\Big)\Big|\\
&\le Cd_{\mathrm{CC}}\Big( \tau_{n^{-1/2}}\big(\Phi(x)\big), \tau_{n^{-1/2}}\big(\Phi_0(x)\big)\Big)
=\frac{C}{n^{1/2}}d_{\mathrm{CC}}\big(\Phi(x), \Phi_0(x) \big), \qquad x \in V. 
\end{aligned}
$$
Since $d_{\mathrm{CC}}\big(\Phi(x), \Phi_0(x) \big)
=d_{\mathrm{CC}}\big(\Phi(\gamma x), \Phi_0(\gamma x) \big)$
for $x \in V$ and $\gamma \in \Gamma$, the function 
$$
x \longmapsto d_{\mathrm{CC}}\big(\Phi(x), \Phi_0(x) \big)
$$
can be regarded as a function defined on the base graph $X_0$. Therefore, we have 
$$
|\mathcal{P}_{n^{-1/2}}f(x)-  \mathcal{P}^H_{n^{-1/2}}f(x)| \le \frac{C}{n^{1/2}} 
\max_{x \in V_0}d_{\mathrm{CC}}\big(\Phi(x), \Phi_0(x) \big), \qquad x \in V. 
$$
Similarly, since $\e^{t\A(\Phi_0)}f$ is also Lipschitz, we have 
$$
| \mathcal{P}^H_{n^{-1/2}}\e^{t\mathcal{A}(\Phi_0)}f(x)
- \mathcal{P}_{n^{-1/2}}\e^{t\mathcal{A}(\Phi_0)}f(x)|_\infty \le \frac{C}{n^{1/2}} 
\max_{x \in V_0}d_{\mathrm{CC}}\big(\Phi(x), \Phi_0(x) \big), \qquad x \in V. 
$$
for some positive constant $C>0$. By putting it all together, we obtain the desired bound. 
This completes the proof. 
\end{proof}

The difference between $\Phi(x)$ and $\Phi_0(x)$
with respect to $d_{\mathrm{CC}}$
is called the {\it corrector}, which is also applied effectively 
to the proof of a functional CLT 
for non-symmetric random walks on $X$ in e.g., \cite[Theorem 2.3]{IKN}. 
We note that the terminology ``corrector'' is frequently used in the context of 
homogenization theory (cf.\,\cite{Kumagai}).


\begin{appendix}
\section{{\bf Proof of Proposition \ref{ergodic-iterate}}}

We here give a proof of Proposition \ref{ergodic-iterate} in the case $N=2$. 

\begin{proof}[Proof of Proposition \ref{ergodic-iterate}] 
Throughout the proof, 
$\la \cdot, \cdot \ra_{\ell^2(X_0)}$ and $\|\cdot\|_{\ell^2(X_0)}$ 
are abbreviated as $\la \cdot, \cdot \ra$ and $\|\cdot\|$, respectively.  
By virtue of the decomposition \eqref{decomposition}, we have
$$
\begin{aligned}
&\frac{1}{n^2}\sum_{k=0}^{n-1}\sum_{\ell=0}^k\mathcal{L}^\ell f(x)\mathcal{L}^{k+1}g(x)\\
&=\frac{1}{2}\Big(1+\frac{1}{n}\Big)\la f, m\ra
\la g, m\ra+ \frac{1}{n^2}\la f, m\ra\sum_{k=1}^{n}k\sum_{j=1}^{K_0-1}\la g, \psi_j\ra\alpha_j^{k}\phi_j(x)\\
&\hspace{1cm}+\frac{1}{n^2}\la f, m\ra\sum_{k=0}^{n-1}
\sum_{\ell=0}^k \LL^{\ell} g_{\ell_{K_0}^2(X_0)}(x)+\frac{1}{n^2}\la g, m\ra\sum_{k=0}^{n-1}\sum_{\ell=0}^k
\sum_{j=1}^{K_0-1}\la f, \psi_j\ra\alpha_j^{\ell}\phi_j(x)\\
&\hspace{1cm}+\frac{1}{n^2}\sum_{k=0}^{n-1}\sum_{\ell=0}^k
\Big(\sum_{j=1}^{K_0-1}\la f, \psi_j\ra\alpha_j^{\ell}\phi_j(x)\Big)
\Big(\sum_{j=1}^{K_0-1}\la g, \psi_j\ra\alpha_j^{k+1}\phi_j(x)\Big)\\
&\hspace{1cm}+\frac{1}{n^2}\sum_{k=0}^{n-1}\mathcal{L}^{k+1}g_{\ell^2_{K_0}(X_0)}(x)
\sum_{\ell=0}^k
\Big(\sum_{j=1}^{K_0-1}\la f, \psi_j\ra\alpha_j^{\ell}\phi_j(x)\Big)
+\frac{1}{n^2}\la g, m\ra\sum_{k=1}^{n}k
\mathcal{L}^{k}f_{\ell^2_{K_0}(X_0)}(x)\\
&\hspace{1cm}+\frac{1}{n^2}\sum_{k=0}^{n-1}
\Big(\sum_{j=1}^{K_0-1}\la g, \psi_j\ra\alpha_j^{k+1}\phi_j(x)\Big)
\sum_{\ell=0}^k \LL^{\ell}f_{\ell^2_{K_0}(X_0)}(x)\nn\\
&\hspace{1cm}
+\frac{1}{n^2}\sum_{k=0}^{n-1}\sum_{\ell=0}^k
 \LL^{\ell}f_{\ell^2_{K_0}(X_0)}(x) \LL^{k+1}g_{\ell^2_{K_0}(X_0)}(x)\\
 &=:\frac{1}{2}\Big(1+\frac{1}{n}\Big)\la f, m\ra
\la g, m\ra+I_1+I_2+I_3+I_4+I_5+I_6+I_7+I_8
\end{aligned}
$$
for $n \in\mathbb{N}$ and $x \in V_0$. 
In particular, the terms $I_1, I_3, I_4$ and $I_5$ are calculated as follows:
$$
\begin{aligned}
I_1&=-\frac{1}{n}\la f, m \ra \sum_{j=1}^{K_0-1}\frac{\alpha_j^{n+1}\la g, \psi_j\ra }{1-\alpha_j}\phi_j(x)+\frac{1}{n^2}\la f, m \ra \sum_{j=1}^{K_0-1}\frac{\alpha_j(1-\alpha_j^n)\la g, \psi_j\ra}{(1-\alpha_j)^2}\phi_j(x),\\
I_3&=\frac{1}{n}\la g, m \ra \sum_{j=1}^{K_0-1}\frac{\la f, \psi_j\ra}{1-\alpha_j}\phi_j(x)
-\frac{1}{n^2}\la g, m \ra \sum_{j=1}^{K_0-1}\frac{\alpha_j(1-\alpha_j^n)\la f, \psi_j\ra}{(1-\alpha_j)^2}\phi_j(x),\\
I_4&=\frac{1}{n^2}\sum_{i, j=1}^{K_0-1}\frac{\la f, \psi_i\ra\la g, \psi_j\ra}{1-\alpha_i}\Big(
\frac{\alpha_j(1-\alpha_j^n)}{1-\alpha_j} - \frac{\alpha_i\alpha_j(1-\alpha_i^n\alpha_j^n)}{1-\alpha_i\alpha_j}\Big)\phi_i(x)\phi_j(x),\\
I_5&=\frac{1}{n^2}\sum_{k=0}^{n-1}\mathcal{L}^{k+1}g_{\ell^2_{K_0}(X_0)}(x)
\sum_{j=1}^{K_0-1}\frac{\la f, \psi_j\ra(1-\alpha_j^{k+1})}{1-\alpha_j}\phi_j(x).
\end{aligned}
$$
We note that the Perron--Frobenius theorem implies 
that there exists some $\lambda \in (0, 1]$
such that $\big\|\LL|_{\ell^2_{K_0}(X_0)}\big\| \leq \lambda$. Hence, we see that 
\begin{align}
\Big\|\sum_{k=0}^{n-1}\LL^{k}f\Big\|
&\leq \sum_{k=0}^{n-1} \lambda^k \|f\|
\leq \frac{\|f\|}{1-\lambda}=O(1), & 
\quad\Big\|\sum_{k=1}^{n}k\LL^{k}f\Big\|
&\leq n\sum_{k=1}^{n} \lambda^k  \|f\|=O(n), \nonumber\\
\Big\|\sum_{k=0}^{n-1}\sum_{\ell=0}^k \LL^{\ell}f\Big\|
&\leq \sum_{k=0}^{n-1}\sum_{\ell=0}^k \lambda^k  \|f\|=O(n) & \label{L-ineq}
\end{align}
for $f \in \ell_{K_0}^2(X_0)$. 
We now set
\begin{align}
A[f, g]_n^{(1)}(x) &=\frac{1}{2}\la f, m \ra \la g, m \ra -\la f, m \ra \sum_{j=1}^{K_0-1}\frac{\alpha_j^{n+1}\la g, \psi_j\ra }{1-\alpha_j}\phi_j(x)\nn\\
&\hspace{1cm}
+\frac{1}{n}\la f, m\ra\sum_{k=0}^{n-1}
\sum_{\ell=0}^k \LL^{\ell} g_{\ell_{K_0}^2(X_0)}(x)\nonumber\\
&\hspace{1cm}+\la g, m \ra \sum_{j=1}^{K_0-1}\frac{\la f, \psi_j\ra}{1-\alpha_j}\phi_j(x)
+\frac{1}{n}\la g, m\ra\sum_{k=1}^{n}k
\mathcal{L}^{k}f_{\ell^2_{K_0}(X_0)}(x),\label{A1}
\end{align}
and 
\begin{align}
A[f, g]_n^{(2)}(x) &=\la f, m \ra \sum_{j=1}^{K_0-1}\frac{\alpha_j(1-\alpha_j^n)\la g, \psi_j\ra}{(1-\alpha_j)^2}\phi_j(x) -\la g, m \ra \sum_{j=1}^{K_0-1}\frac{\alpha_j(1-\alpha_j^n)\la f, \psi_j\ra}{(1-\alpha_j)^2}\phi_j(x) \nonumber\\
&\hspace{1cm}+\sum_{i, j=1}^{K_0-1}\frac{\la f, \psi_i\ra\la g, \psi_j\ra}{1-\alpha_i}\Big(
\frac{\alpha_j(1-\alpha_j^n)}{1-\alpha_j} - \frac{\alpha_i\alpha_j(1-\alpha_i^n\alpha_j^n)}{1-\alpha_i\alpha_j}\Big)\phi_i(x)\phi_j(x)\nonumber\\
&\hspace{1cm}+\sum_{k=0}^{n-1}\mathcal{L}^{k+1}g_{\ell^2_{K_0}(X_0)}(x)
\sum_{j=1}^{K_0-1}\frac{\la f, \psi_j\ra(1-\alpha_j^{k+1})}{1-\alpha_j}\phi_j(x)
\nonumber\\
&\hspace{1cm}+\sum_{k=0}^{n-1}
\Big(\sum_{j=1}^{K_0-1}\la g, \psi_j\ra\alpha_j^{k+1}\phi_j(x)\Big)
\sum_{\ell=0}^k \LL^{\ell}f_{\ell^2_{K_0}(X_0)}(x)\nonumber\\
&\hspace{1cm}+\sum_{k=0}^{n-1}\sum_{\ell=0}^k
 \LL^{\ell}f_{\ell^2_{K_0}(X_0)}(x) \LL^{k+1}g_{\ell^2_{K_0}(X_0)}(x).\label{A2}
\end{align}
By noting \eqref{L-ineq},  $|\alpha_j|\leq 1, \, j=1, 2, \dots, K_0-1$, and an inequality 
$\|fg\|^2 \le |V_0|\|f\|^2\|g\|^2$ 
for $f, g \in \ell^2(X_0)$, 
we conclude that $\big\|A[f, g]_n^{(1)}\big\|=O(1)$ and 
$\big\|A[f, g]_n^{(2)}\big\|=O(1)$ as $n \to \infty$. 
This completes the proof of Proposition \ref{ergodic-iterate}.

\end{proof}

\end{appendix}

\noindent
{\bf Acknowledgement.} 
The author would like to thank Professor Hiroshi Kawabi for 
providing valuable comments which make the present paper more readable. 
He also would like to thank an anonymous referee for 
reading his manuscript carefully and providing helpful comments. 
This work is supported by KAKENHI Grant Number 
No.\,19K23410. 




\end{document}